\documentclass[twoside,11pt]{article}

\usepackage{jmlr2e}
\usepackage{lipsum}
\usepackage{amsfonts}
\usepackage{epstopdf}
\usepackage{algorithm}
\usepackage{algorithmic}
\usepackage{amsopn}
\usepackage{bm}
\usepackage{mathtools}
\usepackage{enumitem}
\usepackage{amsmath}
\usepackage{mathrsfs}
\usepackage{amsfonts}
\usepackage{caption}
\usepackage{subfigure}
\usepackage{color}
\usepackage{float}
\usepackage[normalem]{ulem}
\usepackage{multirow}
\usepackage{lastpage}

\definecolor{blue}{rgb}{0,0,0.9}
\definecolor{red}{rgb}{0.9,0,0}
\definecolor{green}{rgb}{0,0.9,0}

\def\e{\bm{e}}

\ShortHeadings{Fast Algorithm for Computing Wasserstein Barycenters}{Yang, Li, Sun and Toh}
\firstpageno{1}

\graphicspath{{images/}}

\begin{document}

\title{A Fast Globally Linearly Convergent Algorithm for the Computation of Wasserstein Barycenters}

\author{\name Lei Yang \email matylei@nus.edu.sg \\
        \addr Department of Mathematics  \\
              National University of Singapore \\
              10 Lower Kent Ridge Road, Singapore 119076
        \AND
        \name Jia Li \email jiali@stat.psu.edu \\
        \addr Department of Statistics \\
              Pennsylvania State University \\
              University Park, PA 16802, USA
        \AND
        \name Defeng Sun \email defeng.sun@polyu.edu.hk \\
        \addr Department of Applied Mathematics \\
              The Hong Kong Polytechnic University \\
              Hung Hom, Kowloon, Hong Kong
        \AND
        \name Kim-Chuan Toh \email mattohkc@nus.edu.sg \\
        \addr Department of Mathematics and Institute of Operations Research and Analytics \\
              National University of Singapore \\
              10 Lower Kent Ridge Road, Singapore 119076
}

\maketitle

\begin{abstract}
We consider the problem of computing a Wasserstein barycenter for a set of discrete probability distributions with finite supports, which finds many applications in areas such as statistics, machine learning and image processing. When the support points of the barycenter are pre-specified, this problem can be modeled as a linear programming (LP) problem whose size can be extremely large. To handle this large-scale LP, we analyse the structure of its dual problem, which is conceivably more tractable and can be reformulated as a well-structured convex problem with 3 kinds of block variables and a coupling linear equality constraint. We then adapt a symmetric Gauss-Seidel based alternating direction method of multipliers (sGS-ADMM) to solve the resulting dual problem and establish its global convergence and global linear convergence rate. As a critical component for efficient computation, we also show how all the subproblems involved can be solved exactly and efficiently. This makes our method suitable for computing a Wasserstein barycenter on a large-scale data set, without introducing an entropy regularization term as is commonly practiced. In addition, our sGS-ADMM can be used as a subroutine in an alternating minimization method to compute a barycenter when its support points are not pre-specified. Numerical results on synthetic data sets and image data sets demonstrate that our method is highly competitive for solving large-scale Wasserstein barycenter problems, in comparison to two existing representative methods and the commercial software Gurobi.
\end{abstract}

\begin{keywords}
Wasserstein barycenter, discrete probability distribution, semi-proximal ADMM, symmetric Gauss-Seidel
\end{keywords}

\section{Introduction}\label{secintro}

In this paper, we consider the problem of computing the mean of a set of discrete probability distributions under the Wasserstein distance (also known as the optimal transport distance \textit{or} the earth mover's distance). This mean, called the Wasserstein barycenter, is also a discrete probability distribution \citep{ac2011barycenters}. Recently, the Wasserstein barycenter has attracted much attention due to its promising performance in many application areas such as data analysis and statistics \citep{bk2012consistent}, machine learning \citep{cd2014fast,lw2008real,yl2014scaling,ywwl2017fast} and image processing \citep{rpdb2011wasserstein}. For a set of discrete probability distributions with finite support points, a Wasserstein barycenter with its support points being pre-specified can be computed by solving a linear programming (LP) problem \citep{abm2016discrete}. However, the problem size can be extremely large when the number of discrete distributions \textit{or} the number of support points of each distribution is large. Thus,  classical LP methods such as the simplex method and the interior point method are no longer efficient enough or consume too much memory when solving this problem. This motivates the study of fast algorithms for the computation of Wasserstein barycenters; see, for example, \citep{bccnp2015iterative,borgwardt2019an,borgwardt2020improved,coo2015numerical,ccs2018stochastic,cd2014fast,cp2016a,or2015an,s2016stabilized,sgpc2015convolutional,uddgn2018distributed,xwwz2018fast,yl2014scaling,ywwl2017fast}.

One representative approach is to introduce an entropy regularization in the LP and then apply some efficient first-order methods, e.g., the gradient descent method \citep{cd2014fast} and the iterative Bregman projection (IBP) method \citep{bccnp2015iterative}, to solve the regularized problem. These methods can be implemented efficiently and hence are suitable for large-scale data sets. However, they can only return an approximate solution of the LP (due to the entropy regularization) and often suffer from numerical instabilities and very slow convergence speed when the regularization parameter becomes small. The numerical issue can be alleviated by performing some stabilization techniques (e.g., the \textit{log-sum-exp} technique) at the expense of losing some computational efficiency, but the slow speed may not be avoided. Thus, IBP is highly efficient if a rough approximate solution is adequate, as is the case in many learning tasks. However, our aim here is to obtain a high precision solution efficiently. Detailed empirical studies on the pros and cons of IBP are provided by \cite{ywwl2017fast}, specifically, in the scenario when the regularization parameter is reduced to obtain higher precision solutions. It was found that numerical difficulties often occur and the computational efficiency is lost when driving the regularization parameter to smaller values for obtaining more accurate solutions. We will also provide a comparison with IBP in our experiments. Another approach is to consider the LP as a constrained convex optimization problem with a separable structure and then apply some splitting methods to solve it. For example, the alternating direction method of multipliers (ADMM) was adapted in \citep{yl2014scaling}. However, solving the quadratic programming subproblems involved is still highly expensive. Later, \cite{ywwl2017fast} developed a modified Bregman ADMM (BADMM) based on the original one \citep{wb2014bregman} to solve the LP. In this method, all subproblems have closed-form solutions and hence can be solved efficiently. Promising numerical performance was also reported in \citep{ywwl2017fast}. However, this modified Bregman ADMM does not have a convergence guarantee so far.

In this paper, we also consider the LP as a constrained convex problem with multiple blocks of variables and develop an efficient method to solve its dual LP \textit{without} introducing the entropy regularization to modify the objective function. Here, we should mention that although introducing the entropy regularization can give a certain `smooth' solution (that may be favorable in some learning tasks) and lead to the efficient method IBP, it also introduces some blurring in transport plans (see Figure \ref{1dbarycenters}(b)), which may be undesirable in many other applications. The blurred transport plan can be hard to use for other purposes, for example, the recovery of non-mass splitting transport plan \citep{borgwardt2019an}. In contrast, as discussed in \citep{borgwardt2019an}, an exact barycenter computed from the non-regularized LP can have several favorable properties. Therefore, we believe it is important to have an efficient algorithm that can faithfully solve the original LP.  Moreover, the non-regularization-based method can naturally avoid the numerical issues caused by the entropy regularization and thus it is numerically more stable.

Our method is a convergent 3-block ADMM that is designed based on recent progresses in research on convergent multi-block ADMM-type methods for solving convex composite conic programming; see \citep{cst2017efficient,lst2016schur}. It is well known that the classical ADMM was originally proposed to solve a convex problem that contains 2 blocks of variables and a coupling linear equality constraint \citep{gm1976a,gm1975sur}. Specifically, consider
\begin{equation}\label{admmmodel}
\min\limits_{\bm{x}_1\in\mathbb{R}^{n_1},\,\bm{x}_2\in\mathbb{R}^{n_2}} ~f_1(\bm{x}_1) + f_2(\bm{x}_2) \quad \mathrm{s.t.} \quad \mathcal{A}_1(\bm{x}_1) + \mathcal{A}_2(\bm{x}_2) = \bm{b},
\end{equation}
where $f_1:\mathbb{R}^{n_1}\to(-\infty, \infty]$ and $f_2:\mathbb{R}^{n_2}\to(-\infty, \infty]$ are proper closed convex functions, $\mathcal{A}_1:\mathbb{R}^{n_1} \rightarrow \mathbb{R}^{m}$ and $\mathcal{A}_2:\mathbb{R}^{n_2} \rightarrow \mathbb{R}^{m}$ are linear operators, $\bm{b}\in\mathbb{R}^m$ is a given vector. The iterative scheme of the ADMM for solving problem \eqref{admmmodel} is given as follows:
\begin{equation*}
\left\{\begin{aligned}
&\bm{x}_1^{k+1} \in \arg\min
\limits_{\bm{x}_1\in\mathbb{R}^{n_1}}\left\{\mathcal{L}_{\beta}(\bm{x}_1,\,\bm{x}_2^k,\,\bm{\lambda}^k)\right\},\\
&\bm{x}_2^{k+1} \in \arg\min
\limits_{\bm{x}_2\in\mathbb{R}^{n_2}}\left\{\mathcal{L}_{\beta}(\bm{x}_1^{k+1},\,\bm{x}_2,\,\bm{\lambda}^k)\right\},\\
&\bm{\lambda}^{k+1} = \bm{\lambda}^k + \tau\beta\big(\mathcal{A}_1(\bm{x}_1^{k+1})+\mathcal{A}_2(\bm{x}_2^{k+1})-\bm{b}\big),
\end{aligned}\right.
\end{equation*}
where $\tau \in (0,\frac{\sqrt{5}+1}2)$ is the dual step-size and $\mathcal{L}_{\beta}$ is the augmented Lagrangian function for \eqref{admmmodel} defined as
\begin{equation*}
\mathcal{L}_{\beta}(\bm{x}_1,\,\bm{x}_2,\,\bm{\lambda}) := f_1(\bm{x}_1) + f_2(\bm{x}_2) + \langle \bm{\lambda}, \,\mathcal{A}_1(\bm{x}_1) + \mathcal{A}_2(\bm{x}_2) - \bm{b} \rangle
+{\textstyle\frac{\beta}{2}}\|\mathcal{A}_1(\bm{x}_1) + \mathcal{A}_2(\bm{x}_2) - \bm{b}\|^2
\end{equation*}
with $\beta>0$ being the penalty parameter. Under some mild conditions, the sequence $\{(\bm{x}_1^k, \,\bm{x}_2^k)\}$ generated by the above scheme can be shown to converge to an optimal solution of problem \eqref{admmmodel}. The above 2-block ADMM can be simply extended to a multi-block ADMM of the sequential Gauss-Seidel order for solving a convex problem with more than 2 blocks of variables. However, it has been shown in \citep{chyy2016the} that such a  directly extended ADMM may not converge when applied to a  problem with 3 \textit{or} more blocks of variables. This has motivated many researchers to develop various convergent variants of the ADMM for convex problems with more than 2 blocks of variables; see, for example, \citep{cst2017efficient,clst2018equivalence,hty2012alternating,lst2015a,lst2016schur,sty2015convergent}. Among them, the Schur complement based convergent semi-proximal ADMM (sPADMM) was proposed by \cite{lst2016schur} to solve a large class of linearly constrained convex problems with multiple blocks of variables, whose objective can be the sum of two proper closed convex functions and a finite number of convex quadratic or linear functions. This method modified the original ADMM by performing one more \textit{forward Gauss-Seidel sweep} after updating the block of variables corresponding to the nonsmooth function in the objective. With this novel strategy, \cite{lst2016schur} showed that their method can be reformulated as a 2-block sPADMM with specially designed semi-proximal terms and its convergence is thus guaranteed from that of the 2-block sPADMM; see \citep[Appendix B]{fpst2013hankel}. Later, this method was generalized to the inexact symmetric Gauss-Seidel based ADMM (sGS-ADMM) for more general convex problems \citep{cst2017efficient,lst2018qsdpnal}. The numerical results reported in \citep{cst2017efficient,lst2016schur,lst2018qsdpnal} also showed that the sGS-ADMM always performs much better than the possibly non-convergent directly extended ADMM. In addition, as the sGS-ADMM is equivalent to a 2-block sPADMM with specially designed proximal terms, the linear convergence rate of the sGS-ADMM can also be derived based on the linear convergence rate of the 2-block sPADMM under some mild conditions; more details can be found in \citep[Section 4.1]{hsz2017linear}.

Motivated by the above studies, in this paper, we adapt the sGS-ADMM to compute a Wasserstein barycenter by solving the dual problem of the original primal LP. The contributions of this paper are listed as follows:
\begin{itemize}
\item[1.] We derive the dual problem of the original primal LP and characterize the properties of their optimal solutions; see Proposition \ref{existopt}. The resulting dual problem is our target problem, which is reformulated as a linearly constrained convex problem containing 3 blocks of variables with a carefully delineated separable structure designed for efficient computations. We should emphasize again that we do not introduce the entropic or quadratic regularization to modify the LP so as to make it computationally more tractable. This is in contrast to many existing works that primarily focus on solving an approximation of the original LP arising from optimal transport related problems; see, for example, \citep{bccnp2015iterative,c2013sinkhorn,cd2014fast,dpr2018regularized,es2018quadratically}.

\item[2.] We apply the sGS-ADMM to solve the resulting dual problem and analyze its global convergence as well as its global linear convergence rate without any condition; see Theorems \ref{conthm} and \ref{thmconrate}. As a critical component of the paper, we also develop essential numerical strategies to show how all the subproblems in our method can be solved efficiently and that the subproblems at each step can be computed in parallel. This makes our sGS-ADMM highly suitable for computing Wasserstein barycenters on a large-scale data set.

\item[3.] We conduct rigorous numerical experiments on synthetic data sets and image data sets to evaluate the performance of our sGS-ADMM in comparison to existing state-of-the-art methods (IBP and BADMM) and the highly powerful commercial solver Gurobi. The computational results show that our sGS-ADMM is highly competitive compared to IBP and BADMM, and is also able to outperform Gurobi in terms of the computational time for solving large-scale LPs arising from Wasserstein barycenter problems.

\end{itemize}

The rest of this paper is organized as follows. In Section \ref{sectpro}, we describe the basic problem of computing Wasserstein barycenters and derive its dual problem. In Section \ref{sectalg}, we adapt the sGS-ADMM to solve the resulting dual problem and present the efficient implementations for each step that are crucial in making our method competitive. The convergence analysis of the sGS-ADMM is presented in Section \ref{sectcon}. Finally, numerical results are presented in Section \ref{sectnum}, with some concluding remarks given in Section \ref{secconc}.

\vspace{2mm}
\textit{\textbf{Notation and Preliminaries.}} In this paper, we present scalars, vectors and matrices in lower case letters, bold lower case letters and upper case letters, respectively. We use $\mathbb{R}$, $\mathbb{R}^n$, $\mathbb{R}^n_+$ and $\mathbb{R}^{m\times n}$ to denote the set of real numbers, $n$-dimensional real vectors, $n$-dimensional real vectors with nonnegative entries and $m\times n$ real matrices, respectively. For a vector $\bm{x}$, $x_i$ denotes its $i$-th entry, $\|\bm{x}\|$ denotes its Euclidean norm, $\|\bm{x}\|_p$ denotes its $\ell_p$-norm ($p\geq1$) defined by $\|\bm{x}\|_p:=\left( \sum_{i=1}^{n} |x_{i}|^p \right)^{\frac{1}{p}}$ and $\|\bm{x}\|_T:=\sqrt{\langle\bm{x},\,T\bm{x}\rangle}$ denotes its weighted norm associated with the symmetric positive semidefinite matrix $T$. For a matrix $X$, $x_{ij}$ denotes its $(i,j)$-th entry, $X_{i:}$ denotes its $i$-th row, $X_{:j}$ denotes its $j$-th column, $\|X\|_F$ denotes its Fr\"{o}benius norm and $\mathrm{vec}(X)$ denotes the vectorization of $X$. We also use $\bm{x}\geq0$ and $X \geq 0$ to denote $x_i\geq0$ for all $i$ and $x_{ij}\geq 0$ for all $(i,j)$. The identity matrix of size $n \times n$ is denoted by $I_n$. For any $X_1\in\mathbb{R}^{m \times n_1}$ and $X_2\in\mathbb{R}^{m \times n_2}$, $[X_1, X_2]\in\mathbb{R}^{m \times (n_1+n_2)}$ denotes the matrix obtained by horizontally concatenating $X_1$ and $X_2$. For any $Y_1\in\mathbb{R}^{m_1 \times n}$ and $Y_2\in\mathbb{R}^{m_2 \times n}$, $[Y_1; Y_2]\in\mathbb{R}^{(m_1+m_2) \times n}$ denotes the matrix obtained by vertically concatenating $Y_1$ and $Y_2$. For any $X\in\mathbb{R}^{m \times n}$ and $Y\in\mathbb{R}^{m' \times n'}$, the Kronecker product $X\otimes Y$ is defined as
\begin{equation*}
X\otimes Y =
\begin{bmatrix*}[c]
x_{11}Y   & \cdots & x_{1n}Y  \\
\vdots    &        & \vdots     \\
x_{m1}Y   & \cdots & x_{mn}Y \\
\end{bmatrix*}.
\end{equation*}

For an extended-real-valued function $f: \mathbb{R}^{n} \rightarrow [-\infty,\infty]$, we say that it is \textit{proper} if $f(\bm{x}) > -\infty$ for all $\bm{x} \in \mathbb{R}^{n}$ and its domain ${\rm dom}\,f:=\{\bm{x} \in \mathbb{R}^{n} : f(\bm{x}) < \infty\}$ is nonempty. A proper function $f$ is said to be closed if it is lower semicontinuous. Assume that $f: \mathbb{R}^{n} \rightarrow (-\infty,\infty]$ is a proper and closed convex function. The subdifferential of $f$ at $\bm{x}\in{\rm dom}\,f$ is defined by $\partial f(\bm{x}):=\{\bm{d}\in\mathbb{R}^n: f(\bm{y}) \geq f(\bm{x}) + \langle \bm{d}, \,\bm{y}-\bm{x} \rangle, ~\forall\,\bm{y}\in\mathbb{R}^n\}$ and its conjugate function $f^*: \mathbb{R}^{n} \rightarrow (-\infty,\infty]$ is defined by $f^*(\bm{y}):=\sup\{\langle \bm{y},\,\bm{x}\rangle-f(\bm{x}) : \bm{x}\in\mathbb{R}^n\}$. For any $\bm{x}$ and $\bm{y}$, it follows from \citep[Theorem 23.5]{r1970convex} that
\begin{equation}\label{subeqv}
\bm{y} \in \partial f(\bm{x}) ~~\Longleftrightarrow ~~\bm{x} \in \partial f^*(\bm{y}).
\end{equation}
For any $\nu>0$, the proximal mapping of $\nu f$ at $\bm{y}$ is defined by
$$\mathrm{Prox}_{\nu f}(\bm{y}) := \arg\min_{\bm{x}} \left\{f(\bm{x}) + \frac{1}{2\nu}\|\bm{x} - \bm{y}\|^2\right\}.$$
For a closed convex set $\mathcal{X}\subseteq\mathbb{R}^{n}$, its indicator function $\delta_{\mathcal{X}}$ is defined by $\delta_{\mathcal{X}}(\bm{x})=0$ if $\bm{x}\in\mathcal{X}$ and $\delta_{\mathcal{X}}(\bm{x})=+\infty$ otherwise. Moreover, we use $\mathrm{Pr}_{\mathcal{X}}(\bm{y})$ to denote the projection of $\bm{y}$ onto a closed convex set $\mathcal{X}$. It is easy to see that $\mathrm{Pr}_{\mathcal{X}}(\cdot)\equiv\mathrm{Prox}_{\delta_{\mathcal{X}}}(\cdot)$.

In the following, a discrete probability distribution $\mathcal{P}$ with finite support points is specified by $\{(a_i,\,\bm{q}_i)\in\mathbb{R}_{+}\times\mathbb{R}^d: i = 1, \cdots, m\}$, where $\{\bm{q}_1, \cdots, \bm{q}_m\}$ are the support points or vectors and $\{a_1, \cdots, a_m\}$ are the associated probabilities or weights satisfying $\sum^{m}_{i=1}a_i=1$ and $a_i\geq0$, $i=1,\cdots,m$. We also use $\Xi^p(\mathbb{R}^d)$ to denote the set of all discrete probability distributions on $\mathbb{R}^d$ with finite $p$-th moment.

\section{Problem Statement}\label{sectpro}

In this section, we briefly recall the Wasserstein distance and describe the problem of computing a Wasserstein barycenter for a set of discrete probability distributions with finite support points. We refer interested readers to
\citep[Chapter 6]{v2008optimal} for more details on the Wasserstein distance and to \citep{ac2011barycenters,abm2016discrete} for more details on the Wasserstein barycenter.

Given two discrete distributions $\mathcal{P}^{(u)}=\{(a_i^{(u)},\,\bm{q}_i^{(u)}): i = 1, \cdots, m_u\}$ and $\mathcal{P}^{(v)}=\{(a_i^{(v)}$, $\bm{q}_i^{(v)}): i = 1, \cdots, m_v\}$, the $p$-Wasserstein distance between $\mathcal{P}^{(u)}$ and $\mathcal{P}^{(v)}$ is defined by
$$\mathcal{W}_p(\mathcal{P}^{(u)}, \,\mathcal{P}^{(v)}):=\sqrt[p]{\mathrm{v}^*},$$
where $p\geq1$ (commonly chosen to be 1 \textit{or} 2) and $\mathrm{v}^*$ is the optimal objective value of the following linear programming problem:
\begin{equation*}
\mathrm{v}^* := \min_{\pi_{ij} \geq 0}\left\{\sum_{i}^{m_u}\sum_{j}^{m_v} \pi_{ij} \|\bm{q}_i^{(u)} - \bm{q}_j^{(v)}\|_p^p ~:~
\begin{aligned}
&{\textstyle\sum^{m_u}_{i=1}}\,\pi_{ij} = a_j^{(v)}, ~~j = 1, \cdots, m_v  \\
&{\textstyle\sum^{m_v}_{j=1}}\,\pi_{ij} = a_i^{(u)}, ~~i = 1, \cdots, m_u
\end{aligned}\right\}.
\end{equation*}
Then, given a set of discrete probability distributions $\{\mathcal{P}^{(t)}\}_{t=1}^N$ with $\mathcal{P}^{(t)}=\{(a_i^{(t)},\,\bm{q}_i^{(t)}): i = 1, \cdots, m_t\}$, a $p$-Wasserstein barycenter $\mathcal{P}:=\{(w_i,\,\bm{x}_i): i = 1, \cdots, m\}$ with $m$ support points is an optimal solution of the following problem
\begin{equation*}
\min\limits_{\mathcal{P}\in\Xi^p(\mathbb{R}^d)}~{\textstyle\sum^N_{t=1}}\gamma_t\big{(}\mathcal{W}_p(\mathcal{P}, \,\mathcal{P}^{(t)})\big{)}^p
\end{equation*}
for given weights $(\gamma_1, \cdots, \gamma_N)$ satisfying $\sum^{N}_{t=1}\gamma_t=1$ and $\gamma_t>0$, $t=1,\cdots,N$. It is worth noting that the number of support points of the true barycenter is generally unknown. In theory, for $p=2$, there exists a sparse barycenter whose number of support points is upper bounded by $\sum^N_{t=1}m_t - N +1$;
see \citep[Theorem 2]{abm2016discrete}. In practice, one usually chooses $m$ by experience and sets a value that is not less than $m_t$ for $t=1,\cdots,N$. Clearly, a larger $m$ would lead to a larger problem size and hence may require more computational cost, as observed from our experiments. Since the Wasserstein distance itself is defined by a LP, the above problem is then a two-stage optimization problem. Using the definition with some simple manipulations, one can equivalently rewrite the above problem as
\begin{equation}\label{mainpro}
\begin{aligned}
&\min\limits_{\bm{w},\,X,\,\{\Pi^{(t)}\}}
~{\textstyle\sum^N_{t=1}}\,\langle\gamma_t\mathcal{D}(X, \,Q^{(t)}), \,\Pi^{(t)} \rangle  \\
&\hspace{0.65cm}\mathrm{s.t.} \hspace{0.7cm} \Pi^{(t)}\bm{e}_{m_t} = \bm{w}, ~(\Pi^{(t)})^{\top}\bm{e}_{m} = \bm{a}^{(t)},~\Pi^{(t)} \geq 0, ~~t = 1, \cdots, N, \\
&\hspace{1.9cm} \bm{e}^{\top}_m\bm{w} = 1, ~\bm{w} \geq 0,
\end{aligned}
\end{equation}
where
\begin{itemize}
\item $\bm{e}_{m_t}$ (resp. $\bm{e}_{m}$) denotes the $m_t$ (resp. $m$) dimensional vector with all entries being 1;
\item $\bm{w}:=(w_1,\cdots,w_m)^{\top}\in\mathbb{R}_{+}^m$, $X:=[\bm{x}_1,\cdots,\bm{x}_m]\in\mathbb{R}^{d\times m}$;
\item $\bm{a}^{(t)} := (a_1^{(t)},\cdots,a_{m_t}^{(t)})^{\top}\in\mathbb{R}_{+}^{m_t}$, $Q^{(t)}:=\big{[}\bm{q}_1^{(t)},\cdots,\bm{q}_{m_t}^{(t)}\big{]}\in\mathbb{R}^{d\times m_t}$ for $t=1,\cdots,N$;
\item $\Pi^{(t)} = \big{[}\,\pi_{ij}^{(t)}\,\big{]}\in\mathbb{R}^{m \times m_t}$, $\mathcal{D}(X, \,Q^{(t)}):=\big{[}\,\|\bm{x}_i - \bm{q}_j^{(t)}\|_p^p\,\big{]} \in \mathbb{R}^{m \times m_t}$ for $t=1,\cdots,N$.
\end{itemize}
\vspace{2mm}
Note that problem \eqref{mainpro} is a nonconvex problem, where one needs to find the optimal support $X$ and the optimal weight vector $\bm{w}$ of a barycenter simultaneously. However, in many real applications, the support $X$ of a barycenter can be specified empirically from the support points of $\{\mathcal{P}^{(t)}\}_{t=1}^N$. Thus, one only needs to find the weight vector $\bm{w}$ of a barycenter. In view of this, from now on, we assume that the support $X$ is given. Consequently, problem \eqref{mainpro} reduces to the following problem:
\begin{equation}\label{subpro1}
\begin{aligned}
&\min\limits_{\bm{w},\,\{\Pi^{(t)}\}}~{\textstyle\sum^N_{t=1}}\langle D^{(t)}, \,\Pi^{(t)} \rangle  \\
&\hspace{0.5cm}\mathrm{s.t.} \hspace{0.5cm} \Pi^{(t)}\bm{e}_{m_t} = \bm{w}, ~(\Pi^{(t)})^{\top}\bm{e}_{m} = \bm{a}^{(t)},~\Pi^{(t)} \geq 0, ~~t = 1, \cdots, N, \\
&\hspace{1.5cm} \bm{e}^{\top}_m\bm{w} = 1, ~\bm{w} \geq 0,
\end{aligned}
\end{equation}
where $D^{(t)}$ denotes $\gamma_t\mathcal{D}(X, \,Q^{(t)})$ for simplicity\footnote{Our method presented later actually can solve problem \eqref{subpro1} for any given matrices $D^{(1)}, \cdots, D^{(N)}$.}. This is also the main problem studied in \citep{bccnp2015iterative,borgwardt2019an,coo2015numerical,ccs2018stochastic,cd2014fast,cp2016a,or2015an,s2016stabilized,uddgn2018distributed,yl2014scaling,ywwl2017fast} for the computation of Wasserstein barycenters. Moreover, one can easily see that problem \eqref{subpro1} is indeed a large-scale LP containing $(m+m\sum^N_{t=1} m_t)$ nonnegative variables and $(Nm+\sum^N_{t=1}m_t+1)$ equality constraints. For $N=100$, $m=1000$ and $m_t = 1000$ for all $t=1,\cdots,N$, such LP has about $10^8$ nonnegative variables and $2\times 10^5$ equality constraints.

\begin{remark}[\textbf{Practical computational consideration when $\bm{a}^{(t)}$ is sparse}]\label{rem}
Note that any feasible point $(\bm{w},\,\{\Pi^{(t)}\})$ of problem \eqref{subpro1} must satisfy $(\Pi^{(t)})^{\top}\bm{e}_{m} = \bm{a}^{(t)}$ and $\Pi^{(t)} \geq 0$ for any $t=1,\cdots,N$. This implies that if $a_j^{(t)}=0$ for some $1\leq j\leq m_t$ and $1 \leq t \leq N$, then $\pi_{ij}^{(t)}=0$ for all $1 \leq i \leq m$, i.e., all entries in the $j$-th column of $\Pi^{(t)}$ are zeros. Based on this fact, one can verify the following statements.
\begin{itemize}
\item For any optimal solution $(\bm{w}^*,\,\{\Pi^{(t),*}\})$ of problem \eqref{subpro1}, the point $(\bm{w}^*,\,\{\Pi^{(t),*}_{\mathcal{J}_t}\})$ is also an optimal solution of the following problem \begin{equation}\label{redsubpro1}
    \begin{aligned}
    &\min\limits_{\bm{w},\,\{\widehat{\Pi}^{(t)}\}}~{\textstyle\sum^N_{t=1}}\langle D^{(t)}, \,\widehat{\Pi}^{(t)} \rangle  \\
    &\hspace{0.5cm}\mathrm{s.t.} \hspace{0.5cm} \widehat{\Pi}^{(t)}\bm{e}_{m'_t} = \bm{w}, ~(\widehat{\Pi}^{(t)})^{\top}\bm{e}_{m} = \bm{a}^{(t)}_{\mathcal{J}_t},~\widehat{\Pi}^{(t)} \geq 0, ~~t = 1, \cdots, N, \\
    &\hspace{1.5cm} \bm{e}^{\top}_m\bm{w} = 1, ~\bm{w} \geq 0,
    \end{aligned}
    \end{equation}
    where $\mathcal{J}_t$ denotes the support set of $\bm{a}^{(t)}$, i.e., $\mathcal{J}_t:=\{\,j : a_j^{(t)}\neq0\,\}$, $m'_t$ denotes the cardinality of $\mathcal{J}_t$, $\bm{a}^{(t)}_{\mathcal{J}_t}\in\mathbb{R}^{m'_t}$ denotes the subvector of $\bm{a}^{(t)}$ obtained by selecting the entries indexed by $\mathcal{J}_t$ and $\Pi^{(t),*}_{\mathcal{J}_t}\in\mathbb{R}^{m \times m'_t}$ denotes the submatrix of $\Pi^{(t),*}$ obtained by selecting the columns indexed by $\mathcal{J}_t$.

\item For any optimal solution $(\bm{w}^*,\,\{\widehat{\Pi}^{(t),*}\})$ of problem \eqref{redsubpro1}, the point $(\bm{w}^*,\,\{\Pi^{(t),*}\})$ obtained by setting $\Pi^{(t),*}_{\mathcal{J}_t}=\widehat{\Pi}^{(t),*}$ and $\Pi^{(t),*}_{\mathcal{J}_t^c}=0$ is also an optimal solution of problem \eqref{subpro1}, where $\mathcal{J}_t^c:=\{\,j : a_j^{(t)}=0\,\}$.
\end{itemize}
\vspace{2mm}
Therefore, one can obtain an optimal solution of problem \eqref{subpro1} by computing an optimal solution of problem \eqref{redsubpro1}. Note that the size of problem \eqref{redsubpro1} can be much smaller than that of problem \eqref{subpro1} when each $\bm{a}^{(t)}$ is sparse, i.e., $m'_t \ll m_t$. Thus, solving problem \eqref{redsubpro1} can reduce the computational cost and save memory in practice. Since problem \eqref{redsubpro1} takes the same form as problem \eqref{subpro1}, we only consider problem \eqref{subpro1} in the following.
\end{remark}

For notational simplicity, let $\Delta_m:=\{\bm{w}\in\mathbb{R}^m : \bm{e}^{\top}_m\bm{w} = 1, ~\bm{w} \geq 0\}$ and $\delta^t_{+}$ be the indicator function over $\{\Pi^{(t)}\in\mathbb{R}^{m\times m_t}:\Pi^{(t)}\geq0\}$ for each $t=1,\cdots,N$. By enforcing the constraints $\bm{w}\in\Delta_m$ and $\Pi^{(t)}\geq0$, $t=1,\cdots,N$ via adding the corresponding indicator functions in the objective, problem \eqref{subpro1} can be equivalently written as
\begin{equation}\label{repro}
\begin{aligned}
&\min\limits_{\bm{w},\,\{\Pi^{(t)}\}}~\delta_{\Delta_m}(\bm{w}) + {\textstyle\sum^N_{t=1}} \delta^t_{+}(\Pi^{(t)})
+ {\textstyle\sum^N_{t=1}}\langle D^{(t)}, \,\Pi^{(t)} \rangle  \\
&\hspace{0.45cm}\mathrm{s.t.} \hspace{0.5cm} \Pi^{(t)}\bm{e}_{m_t} = \bm{w}, ~(\Pi^{(t)})^{\top}\bm{e}_{m} = \bm{a}^{(t)}, ~~t = 1, \cdots, N.
\end{aligned}
\end{equation}

We next derive the dual problem of \eqref{repro} (hence \eqref{subpro1}). To this end, we write down the Lagrangian function associated with \eqref{repro} as follows:
\begin{equation}\label{lagfun1}
\begin{aligned}
&\Upsilon\big{(}\bm{w},\{\Pi^{(t)}\}; \{\bm{y}^{(t)}\}, \{\bm{z}^{(t)}\}\big{)}
:=\delta_{\Delta_m}(\bm{w}) + \sum^N_{t=1} \delta^t_{+}(\Pi^{(t)}) + \sum^N_{t=1} \langle D^{(t)}, \,\Pi^{(t)} \rangle \\
&\qquad + \sum^N_{t=1} \langle \bm{y}^{(t)}, \,\Pi^{(t)}\bm{e}_{m_t}- \bm{w}\rangle + \sum^N_{t=1}\langle \bm{z}^{(t)}, \,(\Pi^{(t)})^{\top}\bm{e}_m-\bm{a}^{(t)}\rangle,
\end{aligned}
\end{equation}
where $\bm{y}^{(t)} \in \mathbb{R}^m$, $\bm{z}^{(t)}\in\mathbb{R}^{m_t}$, $t = 1, \cdots, N$ are multipliers. Then, the dual problem of \eqref{repro} is given by
\begin{equation}\label{dualorg}
\max\limits_{\{\bm{y}^{(t)}\}, \{\bm{z}^{(t)}\}}\,\min\limits_{\bm{w},\{\Pi^{(t)}\}}\,\Upsilon\big{(}\bm{w},\{\Pi^{(t)}\}; \{\bm{y}^{(t)}\}, \{\bm{z}^{(t)}\}\big{)}.
\end{equation}
Observe that
\begin{equation*}
\begin{aligned}
&~~\min\limits_{\bm{w},\,\{\Pi^{(t)}\}}\Upsilon\big{(}\bm{w},\{\Pi^{(t)}\}; \{\bm{y}^{(t)}\}, \{\bm{z}^{(t)}\}\big{)} \\
&=\min\limits_{\bm{w},\,\{\Pi^{(t)}\}} \left\{ \begin{aligned}
&\delta_{\Delta_m}(\bm{w}) - \langle {\textstyle \sum^N_{t=1}} \bm{y}^{(t)}, \,\bm{w} \rangle + {\textstyle\sum^N_{t=1}}\big{(}\delta_{+}^t(\Pi^{(t)}) + \langle D^{(t)} + \bm{y}^{(t)}\bm{e}^{\top}_{m_t} + \bm{e}_{m}(\bm{z}^{(t)})^{\top}, \,\Pi^{(t)}\rangle\big{)} \\
&~~ -{\textstyle\sum^N_{t=1}}\langle \bm{z}^{(t)}, \,\bm{a}^{(t)}\rangle
\end{aligned}\right\} \\
&= \left\{\begin{aligned}
&- \delta^*_{\Delta_m}\big{(}{\textstyle\sum^N_{t=1}}\bm{y}^{(t)}\big{)} - {\textstyle\sum^N_{t=1}}\langle \bm{z}^{(t)}, \,\bm{a}^{(t)}\rangle, \quad \mathrm{if} ~ D^{(t)} + \bm{y}^{(t)}\bm{e}^{\top}_{m_t} + \bm{e}_{m}(\bm{z}^{(t)})^{\top} \geq 0, ~t = 1, \cdots, N, \\
&-\infty, \quad \mathrm{otherwise},
\end{aligned}\right.
\end{aligned}
\end{equation*}
where $\delta^*_{\Delta_m}$ is the conjugate of $\delta_{\Delta_m}$. Thus, \eqref{dualorg} is equivalent to
\begin{equation*}
\begin{aligned}
&\min\limits_{\{\bm{y}^{(t)}\}, \,\{\bm{z}^{(t)}\}}~ \delta^*_{\Delta_m}\big{(}{\textstyle\sum^N_{t=1}\bm{y}^{(t)}}\big{)} + {\textstyle\sum^N_{t=1}}\langle \bm{z}^{(t)}, \,\bm{a}^{(t)}\rangle \\
&\hspace{0.7cm}\mathrm{s.t.} \hspace{0.8cm} D^{(t)} + \bm{y}^{(t)}\bm{e}^{\top}_{m_t} + \bm{e}_{m}(\bm{z}^{(t)})^{\top} \geq 0, ~~t = 1, \cdots,N.
\end{aligned}
\end{equation*}
By introducing auxiliary variables $\bm{u}, V^{(1)}, \cdots, V^{(N)}$, we can further reformulate the above problem as
\begin{equation}\label{subdualpro}
\begin{aligned}
&\min\limits_{\bm{u}, \,\{V^{(t)}\}, \,\{\bm{y}^{(t)}\}, \,\{\bm{z}^{(t)}\}}  \delta^*_{\Delta_m}(\bm{u}) + {\textstyle\sum^N_{t=1}} \delta_{+}^t(V^{(t)}) + {\textstyle\sum^N_{t=1}}\langle \bm{z}^{(t)}, \,\bm{a}^{(t)}\rangle \\
&\hspace{1.35cm}\mathrm{s.t.} \hspace{1.3cm} {\textstyle\sum^N_{t=1}}\bm{y}^{(t)} - \bm{u} = 0,  \\
&\hspace{3.2cm} V^{(t)} - D^{(t)} - \bm{y}^{(t)}\bm{e}^{\top}_{m_t} - \bm{e}_{m}(\bm{z}^{(t)})^{\top} = 0, ~~t = 1, \cdots,N.
\end{aligned}
\end{equation}
Note that problem \eqref{subdualpro} can be viewed as a linearly constrained convex problem with 3 blocks of variables grouped as $\left(\bm{u}, \,\{V^{(t)}\}\right)$, $\{\bm{y}^{(t)}\}$ and $\{\bm{z}^{(t)}\}$, whose objective is nonsmooth only with respect to $\left(\bm{u}, \,\{V^{(t)}\}\right)$ and linear with respect to the other two. Thus, this problem exactly falls into the class of convex problems for which the sGS-ADMM is applicable; see \citep{cst2017efficient,lst2016schur}. Then, it is natural to adapt the sGS-ADMM for solving problem \eqref{subdualpro}, which is presented in the next section.

\begin{remark}[\textbf{2-block ADMM for solving \eqref{subpro1}}]
It is worth noting that one can also apply the 2-block ADMM to solve the primal problem \eqref{subpro1} by introducing some proper auxiliary variables. For example, one can consider the following equivalent reformulation of \eqref{subpro1}:
\begin{equation*}
\begin{aligned}
&\min\limits_{\bm{w}, \{\Pi^{(t)}\}, \{\Gamma^{(t)}\}}~\delta_{\Delta_m}(\bm{w}) + {\textstyle\sum^N_{t=1}} \delta_{\Delta_{\Pi^{(t)}}}(\Pi^{(t)}) + {\textstyle\sum^N_{t=1}} \langle D^{(t)}, \,\Pi^{(t)} \rangle  \\
&\hspace{0.9cm} \mathrm{s.t.} \hspace{1cm} \Pi^{(t)} = \Gamma^{(t)}, \quad \Gamma^{(t)}\bm{e}_{m_t} = \bm{w}, \quad t = 1, \cdots, N,
\end{aligned}
\end{equation*}
where $\Delta_{\Pi^{(t)}} := \{\Pi^{(t)}\in\mathbb{R}^{m\times m_t} : (\Pi^{(t)})^{\top}\bm{e}_m = \bm{a}^{(t)}, ~\Pi^{(t)}\geq0\}$. Then, the 2-block ADMM can be readily applied with $(\bm{w}, \Pi^{(1)}, \cdots, \Pi^{(N)})$ being one block and $(\Gamma^{(1)}, \cdots, \Gamma^{(N)})$ as the other. This 2-block ADMM avoids the need to solve quadratic programming subproblems and hence is more efficient than the one used in \citep{yl2014scaling}. However, it needs to compute the projection onto the $m$-dimensional simplex $(1+\sum^N_{t=1}m_t)$ times when solving the $(\bm{w}, \Pi^{(1)}, \cdots, \Pi^{(N)})$-subproblem in each iteration. This is still time-consuming when $N$ or $m_t$ is large. Thus, this 2-block ADMM is also not efficient enough for solving large-scale problems.  In addition, we have adapted the 2-block ADMM for solving other reformulations of \eqref{subpro1}, but they all perform worse than our sGS-ADMM to be presented later. Hence, we will no longer consider ADMM-type methods for solving the primal problem \eqref{subpro1} or its equivalent variants in this paper.
\end{remark}

\section{sGS-ADMM for Computing Wasserstein Barycenters}\label{sectalg}

In this section, we present the sGS-ADMM for solving problem \eqref{subdualpro}. First, we write down the Lagrangian function associated with \eqref{subdualpro} as follows:
\begin{equation}\label{lagfun2}
\begin{aligned}
&~~\widetilde{\Upsilon}\big{(}\bm{u}, \{V^{(t)}\}, \{\bm{y}^{(t)}\}, \{\bm{z}^{(t)}\}; \bm{\lambda}, \{\Lambda^{(t)}\}\big{)} \\
&= \delta^*_{\Delta_m}(\bm{u}) + {\textstyle\sum^N_{t=1}}\delta_{+}^t(V^{(t)}) + {\textstyle\sum^N_{t=1}}\langle \bm{z}^{(t)}, \,\bm{a}^{(t)}\rangle + \langle \bm{\lambda}, \,{\textstyle\sum^N_{t=1}}\bm{y}^{(t)} - \bm{u}\rangle \\
&~~\quad + {\textstyle\sum^N_{t=1}} \langle \Lambda^{(t)}, \,V^{(t)} - D^{(t)} - \bm{y}^{(t)}\bm{e}^{\top}_{m_t} - \bm{e}_{m}(\bm{z}^{(t)})^{\top}\rangle,
\end{aligned}
\end{equation}
where $\bm{\lambda} \in \mathbb{R}^m$, $\Lambda^{(t)}\in\mathbb{R}^{m \times m_t}$, $t = 1, \cdots, N$ are multipliers. Then, the augmented Lagrangian function associated with \eqref{subdualpro} is
\begin{equation*}
\begin{aligned}
&~~\mathcal{L}_{\beta}\big{(}\bm{u}, \{V^{(t)}\}, \{\bm{y}^{(t)}\}, \{\bm{z}^{(t)}\}; \bm{\lambda}, \{\Lambda^{(t)}\}\big{)} \\
&= \widetilde{\Upsilon}\big{(}\bm{u}, \{V^{(t)}\}, \{\bm{y}^{(t)}\}, \{\bm{z}^{(t)}\}; \bm{\lambda}, \{\Lambda^{(t)}\}\big{)}
+ {\textstyle\frac{\beta}{2}}\|{\textstyle\sum^N_{t=1}}\bm{y}^{(t)} - \bm{u}\|^2 \\
&\qquad + {\textstyle\frac{\beta}{2}\sum^N_{t=1}}\|V^{(t)} - D^{(t)} - \bm{y}^{(t)}\bm{e}^{\top}_{m_t} - \bm{e}_{m}(\bm{z}^{(t)})^{\top}\|_F^2,
\end{aligned}
\end{equation*}
where $\beta>0$ is the penalty parameter. The sGS-ADMM for solving \eqref{subdualpro} is now readily presented in Algorithm \ref{sGSADMMdual}.

\begin{algorithm}[h]
\caption{sGS-ADMM for solving \eqref{subdualpro}}\label{sGSADMMdual}
\textbf{Input:} the penalty parameter $\beta>0$, and the initialization $\bm{u}^0\in\mathbb{R}^m$, $\bm{\lambda}^0\in\mathbb{R}^m$, $\bm{y}^{(t),0} \in \mathbb{R}^m$, $\bm{z}^{(t),0}\in\mathbb{R}^{m_t}$, $V^{(t),0}\in\mathbb{R}_{+}^{m \times m_t}$, $\Lambda^{(t),0}\in\mathbb{R}^{m \times m_t}$, $t = 1, \cdots, N$. Set $k=0$. \\
\textbf{while} a termination criterion is not met, \textbf{do}
\begin{itemize}[leftmargin=1.5cm]
\item[\textbf{Step}] \textbf{1.} Compute
                     \begin{equation*}
                     \big{(}\bm{u}^{k+1},\,\{V^{(t),k+1}\}\big{)} = \arg\min\limits_{\bm{u},\,\{V^{(t)}\}}~\mathcal{L}_{\beta}\big{(}\bm{u}, \{V^{(t)}\}, \{\bm{y}^{(t),k}\}, \{\bm{z}^{(t),k}\};\,\bm{\lambda}^k, \{\Lambda^{(t),k}\}\big{)}.
                     \end{equation*}

\item[\textbf{Step}] \textbf{2a.} Compute
                     \begin{equation*}
                     \{\tilde{\bm{z}}^{(t),k+1}\} = \arg\min\limits_{\{\bm{z}^{(t)}\}}~\mathcal{L}_{\beta}\big{(}\bm{u}^{k+1}, \{V^{(t),k+1}\}, \{\bm{y}^{(t),k}\}, \{\bm{z}^{(t)}\};\,\bm{\lambda}^k, \{\Lambda^{(t),k}\}\big{)}.
                     \end{equation*}

\item[\textbf{Step}] \textbf{2b.} Compute
                     \begin{equation*}
                     \{\bm{y}^{(t),k+1}\} = \arg\min\limits_{\{\bm{y}^{(t)}\}}~\mathcal{L}_{\beta}\big{(}\bm{u}^{k+1}, \{V^{(t),k+1}\}, \{\bm{y}^{(t)}\}, \{\tilde{\bm{z}}^{(t),k+1}\};\,\bm{\lambda}^k, \{\Lambda^{(t),k}\}\big{)}.
                     \end{equation*}

\item[\textbf{Step}] \textbf{2c.} Compute
                     \begin{equation*}
                     \{\bm{z}^{(t),k+1}\} = \arg\min\limits_{\{\bm{z}^{(t)}\}}~\mathcal{L}_{\beta}\big{(}\bm{u}^{k+1}, \{V^{(t),k+1}\}, \{\bm{y}^{(t),k+1}\}, \{\bm{z}^{(t)}\};\,\bm{\lambda}^k, \{\Lambda^{(t),k}\}\big{)}.
                     \end{equation*}

\item[\textbf{Step}] \textbf{3.} Compute
                     \begin{equation*}
                     \begin{aligned}
                     \bm{\lambda}^{k+1} &= \bm{\lambda}^k + \tau\beta\big{(}{\textstyle\sum^N_{t=1}}\bm{y}^{(t),k+1} - \bm{u}^{k+1}\big{)}, \\
                     \Lambda^{(t),k+1} &= \Lambda^{(t),k} + \tau\beta\big{(}V^{(t),k+1} - D^{(t)} - \bm{y}^{(t),k+1}\bm{e}^{\top}_{m_t} - \bm{e}_{m}(\bm{z}^{(t),k+1})^{\top}\big{)}, ~~ t=1,\cdots,N,
                     \end{aligned}
                     \end{equation*}
                     where $\tau\in(0, \frac{1+\sqrt{5}}{2})$ is the dual step-size that is typically set to $1.618$.
\end{itemize}
\textbf{end while}  \\
\textbf{Output}: $\bm{u}^{k+1}$, $\{V^{(t),k+1}\}$, $\{\bm{y}^{(t),k+1}\}$, $\{\bm{z}^{(t),k+1}\}$, $\bm{\lambda}^{k+1}$, $\{\Lambda^{(t),k+1}\}$. \vspace{1mm}
\end{algorithm}

We next show that all subproblems in Algorithm \ref{sGSADMMdual} can be solved efficiently (in fact analytically) and the subproblems in each step can also be computed in parallel. This makes our method highly suitable for solving large-scale problems. The computational details and the efficient implementations in each step are presented as follows.

\begin{itemize}[leftmargin=1cm]

\item[\textbf{Step}] \textbf{1.} Note that $\mathcal{L}_{\beta}$ is actually separable with respect to $\bm{u}, V^{(1)}, \cdots, V^{(N)}$ and hence one can compute $\bm{u}^{k+1}, V^{(1),k+1}, \cdots, V^{(N),k+1}$ independently. Specifically, $\bm{u}^{k+1}$ is obtained by solving
\begin{equation*}
\min \limits_{\bm{u}\in\mathbb{R}^m} \Big\{ \delta^*_{\Delta_m}(\bm{u}) - \langle \bm{\lambda}^k, \, \bm{u}\rangle + {\textstyle\frac{\beta}{2}\|\sum^N_{t=1}\bm{y}^{(t),k} - \bm{u}\|^2} \Big\}.
\end{equation*}
Thus, we have
\begin{equation*}
\begin{aligned}
\bm{u}^{k+1}
&= \mathrm{Prox}_{\beta^{-1}\delta^*_{\Delta_m}} \big{(} \beta^{-1}\bm{\lambda}^k + {\textstyle\sum^N_{t=1}}\bm{y}^{(t),k} \big{)} \\
&= \big{(} \beta^{-1}\bm{\lambda}^k + {\textstyle\sum^N_{t=1}}\bm{y}^{(t),k} \big{)} - \beta^{-1}\, \mathrm{Prox}_{\beta\delta_{\Delta_m}} \big{(}\bm{\lambda}^k+\beta{\textstyle\sum^N_{t=1}}\bm{y}^{(t),k}\big{)},
\end{aligned}
\end{equation*}
where the last equality follows from the Moreau decomposition
\citep[Theorem 14.3(ii)]{bc2011convex}, i.e., $\bm{x} = \mathrm{Prox}_{\nu f^*}(\bm{x}) + \nu \mathrm{Prox}_{f/\nu}(\bm{x}/\nu)$ for any $\nu>0$ and the proximal mapping of $\beta\delta_{\Delta_m}$ can be computed efficiently by the algorithm proposed in \citep{c2016fast} with the complexity of $\mathcal{O}(m)$ that is typically observed in practice. Moreover, for each $t=1,\cdots,N$, $V^{(t),k+1}$ can be computed in parallel by solving
\begin{equation*}
\min \limits_{V^{(t)}} \Big\{ \delta_{+}^t(V^{(t)}) + \langle \Lambda^{(t),k}, \,V^{(t)}\rangle + {\textstyle\frac{\beta}{2}}\|V^{(t)} - D^{(t)} - \bm{y}^{(t),k}\bm{e}^{\top}_{m_t} - \bm{e}_{m}(\bm{z}^{(t),k})^{\top}\|_F^2\Big\}.
\end{equation*}
Then, it is easy to see that
\begin{equation*}
V^{(t),k+1} = \max\big{\{}\widetilde{D}^{(t),k} - \beta^{-1}\Lambda^{(t),k}, \,0\big{\}},
\end{equation*}
where $\widetilde{D}^{(t),k}:=D^{(t)} + \bm{y}^{(t),k}\bm{e}^{\top}_{m_t} + \bm{e}_{m}(\bm{z}^{(t),k})^{\top}$. Note that $\widetilde{D}^{(t),k}$ is already computed for updating $\Lambda^{(t),k}$ in the previous iteration and thus it can be reused in the current iteration. The computational complexity in this step is $\mathcal{O}(Nm+m\sum^N_{t=1}m_t)$. We should emphasize that because the matrices such as $\{\widetilde{D}^{(t),k}\}$, $\{\Lambda^{(t),k}\}$ are large and numerous, even performing simple operations such as adding two such matrices can be time-consuming. Thus, we have paid special attention to arrange the computations in each step of the sGS-ADMM so that matrices computed in one step can be reused for the next step. \vspace{2mm}

\item[\textbf{Step}] \textbf{2a.} Similarly, $\mathcal{L}_{\beta}$ is separable with respect to $\bm{z}^{(1)}, \cdots, \bm{z}^{(N)}$ and then one can also compute $\tilde{\bm{z}}^{(1),k+1}$, $\cdots$, $\tilde{\bm{z}}^{(N),k+1}$ in parallel. For each $t=1,\cdots,N$, $\tilde{\bm{z}}^{(t),k+1}$ is obtained by solving
\begin{equation*}
\min \limits_{\bm{z}^{(t)}} \Big\{ \langle \bm{z}^{(t)}, \,\bm{a}^{(t)}\rangle - \langle \Lambda^{(t),k}, \,\bm{e}_{m}(\bm{z}^{(t)})^{\top}\rangle + {\textstyle\frac{\beta}{2}}\|V^{(t),k+1} - D^{(t)} - \bm{y}^{(t),k}\bm{e}^{\top}_{m_t} - \bm{e}_{m}(\bm{z}^{(t)})^{\top}\|_F^2\Big\}.
\end{equation*}
It is easy to prove that
\begin{equation*}
\begin{aligned}
\tilde{\bm{z}}^{(t),k+1}
&= {\textstyle\frac{1}{m}}\big{(}(V^{(t),k+1})^{\top}\bm{e}_m - (D^{(t)})^{\top}\bm{e}_m
-\big{(}\bm{e}_m^{\top}\bm{y}^{(t),k}\big{)}\bm{e}_{m_t} + \beta^{-1}(\Lambda^{(t),k})^{\top}\bm{e}_m - \beta^{-1}\bm{a}^{(t)} \big{)} \\
&= \bm{z}^{(t),k} - {\textstyle\frac{1}{m}}\big{(} \beta^{-1}\bm{a}^{(t)} + (B^{(t),k})^{\top}\bm{e}_m\big{)},
\end{aligned}
\end{equation*}
where $B^{(t),k} := \widetilde{D}^{(t),k} - \beta^{-1}\Lambda^{(t),k} - V^{(t),k+1} =
\min\{\widetilde{D}^{(t),k} - \beta^{-1}\Lambda^{(t),k}, \,0\}$. Note that $\widetilde{D}^{(t),k} - \beta^{-1}\Lambda^{(t),k}$ has already been computed in \textbf{Step 1} and hence $B^{(t),k}$ can be computed by just a simple $\min(\cdot)$ operation. We note that $\tilde{\bm{z}}^{(t),k+1}$ is computed analytically for all $t=1,\dots,N$, and the computational complexity in this step is $\mathcal{O}((m+1)\sum^N_{t=1}m_t)$. \vspace{2mm}

\item[\textbf{Step}] \textbf{2b.} In this step, one can see that $\bm{y}^{(1)}$, $\cdots$, $\bm{y}^{(N)}$ are coupled in $\mathcal{L}_{\beta}$ (due to the quadratic term $\frac{\beta}{2}\|\sum^N_{t=1}\bm{y}^{(t)}-\bm{u}^{k+1}\|^2$) and hence the problem of minimizing $\mathcal{L}_{\beta}$ with respect to $\bm{y}^{(1)}, \cdots, \bm{y}^{(N)}$ cannot be reduced to $N$ separable subproblems. However, one can still compute them efficiently based on the following observation. Note that $(\bm{y}^{(1),k+1}$, $\cdots$, $\bm{y}^{(N),k+1})$ is obtained by solving
\begin{equation*}
\min\limits_{\bm{y}^{(1)},\cdots,\bm{y}^{(N)}}
\left\{\begin{aligned}
&\Phi^k(\bm{y}^{(1)},\cdots,\bm{y}^{(N)}):=\langle\bm{\lambda}^k, \,{\textstyle\sum^N_{t=1}}\bm{y}^{(t)}\rangle + {\textstyle\frac{\beta}{2}}\|{\textstyle\sum^N_{t=1}}\bm{y}^{(t)} - \bm{u}^{k+1}\|^2 \\
&~ + {\textstyle\sum^N_{t=1}}\big{(}\langle -\Lambda^{(t),k}\bm{e}_{m_t}, \,\bm{y}^{(t)}\rangle + {\textstyle\frac{\beta}{2}}\|V^{(t),k+1} - D^{(t)} - \bm{y}^{(t)}\bm{e}^{\top}_{m_t} - \bm{e}_{m}(\tilde{\bm{z}}^{(t),k+1})^{\top}\|_F^2\big{)}
\end{aligned}\right\}.
\end{equation*}
The gradient of $\Phi^k$ with respect to $\bm{y}^{(t)}$ is
\begin{equation*}
\begin{aligned}
&~~\nabla_{\bm{y}^{(t)}} \Phi^k(\bm{y}^{(1)},\cdots,\bm{y}^{(N)}) \\
&= \bm{\lambda}^k + \beta\big{(}{\textstyle\sum^N_{\ell=1}}\bm{y}^{(\ell)}-\bm{u}^{k+1}\big{)} + \beta \big{(}-\beta^{-1}\Lambda^{(t),k} + D^{(t)}+\bm{y}^{(t)}\bm{e}^{\top}_{m_t}+\bm{e}_{m}(\tilde{\bm{z}}^{(t),k+1})^{\top}-V^{(t),k+1}\big{)}\bm{e}_{m_t} \\
&= \beta{\textstyle\sum^N_{\ell=1}}\bm{y}^{(\ell)} + \beta m_t(\bm{y}^{(t)}-\bm{y}^{(t),k}) + \bm{\lambda}^k - \beta \bm{u}^{k+1} + \beta\big{(}B^{(t),k}+\bm{e}_{m}(\tilde{\bm{z}}^{(t),k+1}-\bm{z}^{(t),k})^{\top}\big{)}\bm{e}_{m_t} \\
&= \beta{\textstyle\sum^N_{\ell=1}}(\bm{y}^{(\ell)}-\bm{y}^{(\ell),k}) + \beta m_t(\bm{y}^{(t)}-\bm{y}^{(t),k}) + \beta\bm{h}^k + \beta\widetilde{B}^{(t),k}\bm{e}_{m_t},
\end{aligned}
\end{equation*}
where $\widetilde{B}^{(t),k}:=B^{(t),k}+\bm{e}_{m}(\tilde{\bm{z}}^{(t),k+1}-\bm{z}^{(t),k})^{\top}$ and $\bm{h}^k:=\beta^{-1}\bm{\lambda}^k - \bm{u}^{k+1} + {\textstyle\sum^N_{\ell=1}}\bm{y}^{(\ell),k}$. It follows from the optimality condition, namely, $\nabla\Phi^k(\bm{y}^{(1),k+1},\cdots,\bm{y}^{(N),k+1})=0$ that, for any $t=1,\cdots,N$,
\begin{equation}\label{addeq1}
{\textstyle\sum^N_{\ell=1}}(\bm{y}^{(\ell),k+1}-\bm{y}^{(\ell),k}) + m_t(\bm{y}^{(t),k+1}-\bm{y}^{(t),k}) + \bm{h}^k + \widetilde{B}^{(t),k}\bm{e}_{m_t} = 0.
\end{equation}
By dividing $m_t$ in \eqref{addeq1} for $t = 1, \cdots, N$, adding all resulting equations and doing some simple algebraic manipulations, one can obtain that
\begin{equation*}
\tilde{\bm{b}}^{k} := {\textstyle\sum^N_{\ell=1}}(\bm{y}^{(\ell),k+1}-\bm{y}^{(\ell),k}) = -\frac{(\sum^N_{\ell=1}m_{\ell}^{-1})\bm{h}^k + \sum^N_{\ell=1}m_{\ell}^{-1}\widetilde{B}^{(\ell),k}\bm{e}_{m_{\ell}}}{1+\sum^N_{\ell=1}m_{\ell}^{-1}}.
\end{equation*}
Then, using this equality and \eqref{addeq1}, we have
\begin{equation*}
\begin{aligned}
\bm{y}^{(t), k+1} &= \bm{y}^{(t), k} - {\textstyle\frac{1}{m_t}}\big{(}\tilde{\bm{b}}^{k} + \bm{h}^k + \widetilde{B}^{(t),k}\bm{e}_{m_t}\big{)}, \quad t = 1, \cdots, N.
\end{aligned}
\end{equation*}
Observe that we can compute $\bm{y}^{(t),k+1}$ analytically for $t=1,\cdots,N$.
In the above computations, one can first compute $\widetilde{B}^{(t),k}\bm{e}_{m_t}$ in parallel for $t=1,\cdots,N$ to obtain $\tilde{\bm{b}}^{k}$. Then, $\bm{y}^{(t), k+1}$ can be computed in parallel for $t=1,\cdots,N$. By using the updating formula for $\tilde{\bm{z}}^{(t),k+1}$ in \textbf{Step 2a}, we have that
$\widetilde{B}^{(t),k} \e_{m_t} = B^{(t),k} \e_{m_t} - \frac{1}{m}\e_m \big(\e_m^T B^{(t),k} \e_{m_t}+ \beta^{-1}
\langle \e_{m_t}, \,\bm{a}^{(t)}\rangle\big)$. Thus, there is no need to form $\widetilde{B}^{(t),k}$ explicitly. The computational complexity in this step is $\mathcal{O}(Nm+m\sum^N_{t=1}m_t)$.

\item[\textbf{Step}] \textbf{2c.} Similar to \textbf{Step 2a}, for each $t=1,\cdots,N$, $\bm{z}^{(t),k+1}$ can be obtained independently by solving
\begin{equation*}
\min \limits_{\bm{z}^{(t)}} \Big\{ \langle \bm{z}^{(t)}, \,\bm{a}^{(t)}\rangle - \langle \Lambda^{(t),k}, \,\bm{e}_{m}(\bm{z}^{(t)})^{\top}\rangle + {\textstyle\frac{\beta}{2}}\|V^{(t),k+1} - D^{(t)} - \bm{y}^{(t),k+1}\bm{e}^{\top}_{m_t} - \bm{e}_{m}(\bm{z}^{(t)})^{\top}\|_F^2 \Big\}
\end{equation*}
and it is easy to show that
\begin{equation*}
\begin{aligned}
\bm{z}^{(t),k+1}
&= \bm{z}^{(t),k} - {\textstyle\frac{1}{m}}( \beta^{-1}\bm{a}^{(t)} + (C^{(t),k})^{\top}\bm{e}_m) \\
&= \bm{z}^{(t),k} - {\textstyle\frac{1}{m}}( \beta^{-1}\bm{a}^{(t)} + (B^{(t),k} + (\bm{y}^{(t),k+1}-\bm{y}^{(t),k})\bm{e}^{\top}_{m_t} )^{\top}\bm{e}_m ) \\
&= \tilde{\bm{z}}^{(t),k+1} - {\textstyle\frac{1}{m}}\big{(}(\bm{y}^{(t),k+1}-\bm{y}^{(t),k})^{\top}\bm{e}_m\big{)}\bm{e}_{m_t},
\end{aligned}
\end{equation*}
where $C^{(t),k}:=D^{(t)} + \bm{y}^{(t),k+1}\bm{e}^{\top}_{m_t} + \bm{e}_{m}(\bm{z}^{(t),k})^{\top} - \beta^{-1}\Lambda^{(t),k} - V^{(t),k+1} = B^{(t),k} + (\bm{y}^{(t),k+1}-\bm{y}^{(t),k})\bm{e}^{\top}_{m_t}$.
In light of the above, one can also compute $\bm{z}^{(t),k+1}$ efficiently. The computational complexity in this step is $\mathcal{O}(Nm+\sum^N_{t=1}m_t)$, which is much smaller than the costs in \textbf{Step 2a} and \textbf{Step 2b}.

\end{itemize}

From the above, together with the update of multipliers in \textbf{Step 3}, one can see that the main computational complexity of our sGS-ADMM at each iteration is $\mathcal{O}(m\sum^N_{t=1}m_t)$.

\begin{remark}[\textbf{Comments on \textbf{Step 2a--2c} in Algorithm \ref{sGSADMMdual}}]
Comparing with the directly extended ADMM, our sGS-ADMM in Algorithm \ref{sGSADMMdual} just has one more update of $\{\tilde{\bm{z}}^{(t),k+1}\}$ in {\bf Step 2a}. This step is actually the key to guarantee the convergence of the algorithm. We shall see in the next section that computing $\left(\{\bm{y}^{(t),k+1}\},\,\{\bm{z}^{(t),k+1}\}\right)$ from {\bf Step 2a--2c} is indeed equivalent to minimizing $\mathcal{L}_{\beta}$ plus a special proximal term simultaneously with respect to $\left(\{\bm{y}^{(t)}\},\,\{\bm{z}^{(t)}\}\right)$. Moreover, the reader may have observed that instead of computing $\{\bm{y}^{(t),k+1}\}$ and $\{\bm{z}^{(t),k+1}\}$ sequentially as in {\bf Step 2a--2c}, one can also  compute $(\{\bm{y}^{(t),k+1}\},\{\bm{z}^{(t),k+1}\})$ simultaneously in one step by solving a huge linear system of equations of dimension $mN + \sum_{t=1}^N m_t$. Unfortunately, for the latter approach, the computation of the solution would require the Cholesky factorization of a huge coefficient matrix, and this approach is not practically viable. In contrast, for our approach in {\bf Step 2a-2c}, we have seen that the solutions can be computed analytically without the need to perform Cholesky factorizations of  large coefficient matrices. This also explains why we have designed the computations as in {\bf Step 2a-2c}.
\end{remark}

\begin{remark}[\textbf{Extension to the free support case}]\label{RemFree}
We briefly discuss the case when the support points of a barycenter are not pre-specified and hence one needs to solve problem \eqref{mainpro} to find a barycenter. Note that problem \eqref{mainpro} can be considered as a problem with $X$ being one variable block and $(\bm{w},\,\{\Pi^{(t)}\})$ being the other block. Then, it is natural to apply an alternating minimization method to solve \eqref{mainpro}. Specifically, with $X$ fixed, problem \eqref{mainpro} indeed reduces to problem \eqref{subpro1} (hence \eqref{repro}), and one can call our sGS-ADMM in Algorithm \ref{sGSADMMdual} as a subroutine to solve it efficiently. On the other hand, with $(\bm{w},\,\{\Pi^{(t)}\})$ fixed, problem \eqref{mainpro} reduces to a simple quadratic optimization problem with respect to $X$ and one can easily obtain the optimal $X^*$ columnwise by computing
\begin{equation*}
\bm{x}_i^* = \left({\textstyle\sum^N_{t=1}\sum^{m_t}_{j=1}}\pi_{ij}^{(t)}\right)^{-1}{\textstyle\sum^N_{t=1}\sum^{m_t}_j}\pi_{ij}^{(t)}\bm{q}_j^{(t)}, \qquad i = 1, \cdots, m.
\end{equation*}
In fact, this alternating minimization strategy has also been used in \citep{cd2014fast,yl2014scaling,ywwl2017fast} to handle the free support case by using their proposed methods as subroutines.
\end{remark}

\section{Convergence Analysis}\label{sectcon}

In this section, we shall establish the global linear convergence of Algorithm \ref{sGSADMMdual} based on the convergence results developed in \citep{fpst2013hankel,hsz2017linear,lst2016schur}. To this end, we first write down the KKT system associated with \eqref{lagfun1} as follows:
\begin{equation}\label{KKTprimal}
\begin{aligned}
&0~\in~ \partial \delta_{\Delta_m}(\bm{w}) - \big{(}{\textstyle\sum^{N}_{t=1}}\bm{y}^{(t)}\big{)}, \\
&0~\in~ \partial \delta_{+}^t (\Pi^{(t)}) + D^{(t)} + \bm{y}^{(t)}\bm{e}^{\top}_{m_t} + \bm{e}_{m}(\bm{z}^{(t)})^{\top}, \quad\forall\,t=1,\cdots,N,  \\
&0~=~ \Pi^{(t)}\bm{e}_{m_t} - \bm{w},  \quad\forall\,t=1,\cdots,N,  \\
&0~=~ (\Pi^{(t)})^{\top}\bm{e}_m - \bm{a}^{(t)}, \quad\forall\,t=1,\cdots,N.
\end{aligned}
\end{equation}
We also write down the KKT system associated with \eqref{lagfun2} as follows:
\begin{equation}\label{KKTdual}
\begin{aligned}
&0~\in~ \partial \delta^*_{\Delta_m}(\bm{u}) - \bm{\lambda},  \\
&0~\in~ \partial \delta_{+}^t (V^{(t)}) + \Lambda^{(t)}, \quad\forall\,t=1,\cdots,N,  \\
&0~=~ \Lambda^{(t)}\bm{e}_{m_t} - \bm{\lambda},  \quad\forall\,t=1,\cdots,N,  \\
&0~=~ (\Lambda^{(t)})^{\top}\bm{e}_m - \bm{a}^{(t)}, \quad\forall\,t=1,\cdots,N,  \\
&0~=~ {\textstyle\sum^N_{t=1}}\bm{y}^{(t)} - \bm{u},  \\
&0~=~ V^{(t)} - D^{(t)} - \bm{y}^{(t)}\bm{e}^{\top}_{m_t} - \bm{e}_{m}(\bm{z}^{(t)})^{\top}, \quad\forall\,t=1,\cdots,N.
\end{aligned}
\end{equation}
Then, we show the existence of optimal solutions of problems \eqref{repro} and \eqref{subdualpro}, and their relations in the following proposition.
\begin{proposition}\label{existopt}
The following statements hold.
\begin{itemize}
\item[{\rm (i)}] The optimal solution of problem \eqref{repro} exists and the solution set of the KKT system \eqref{KKTprimal} is nonempty;

\item[{\rm (ii)}] The optimal solution of problem \eqref{subdualpro} exists and the solution set of the KKT system \eqref{KKTdual} is nonempty;

\item[{\rm (iii)}] If $\big{(}\bm{u}^*, \{V^{(t),*}\}, \{\bm{y}^{(t),*}\}, \{\bm{z}^{(t),*}\}, \bm{\lambda}^*, \{\Lambda^{(t),*}\}\big{)}$ is a solution of the KKT system \eqref{KKTdual}, then $(\bm{u}^*$, $\{V^{(t),*}\}$, $\{\bm{y}^{(t),*}\}$, $\{\bm{z}^{(t),*}\})$ solves \eqref{subdualpro} and $\left(\bm{\lambda}^*, \{\Lambda^{(t),*}\}\right)$ solves \eqref{repro}.
\end{itemize}
\end{proposition}
\begin{proof}
\textit{Statement} (i). Note that \eqref{repro} is equivalent to \eqref{subpro1}. Thus, we only need to show that the optimal solution of \eqref{subpro1} exists. To this end, we first claim that the feasible set of \eqref{subpro1} is nonempty. For simplicity, let
\begin{equation*}
\begin{aligned}
\mathcal{C}_{\mathrm{feas}}
&:=\big{\{}(\bm{w},\,\{\Pi^{(t)}\}): \bm{w}\in\Delta_{m}, \,\Pi^{(t)}\in\Omega^t(\bm{w}),\,t=1,\cdots,N\big{\}}, \\
\Omega^t(\bm{w})
&:=\big{\{}\Pi^{(t)}\in\mathbb{R}^{m \times m_t}: \Pi^{(t)}\bm{e}_{m_t} = \bm{w}, ~(\Pi^{(t)})^{\top}\bm{e}_{m} = \bm{a}^{(t)},~\Pi^{(t)} \geq 0\big{\}}, \quad t = 1, \cdots, N.
\end{aligned}
\end{equation*}
Recall that the simplex $\Delta_{m}$ is nonempty. Then, for any fixed $\bar{\bm{w}}\in\Delta_{m}$, consider the sets $\Omega^1(\bar{\bm{w}})$, $\cdots$, $\Omega^N(\bar{\bm{w}})$. For any $t=1,\cdots,N$, since $\bm{a}^{(t)}$ is the weight vector of the discrete probability distribution $\mathcal{P}^{(t)}$, we have that $\bm{e}_{m_t}^{\top}\bm{a}^{(t)}=1$. Using this fact and $\bm{e}^{\top}_m\bar{\bm{w}}=1$, we have from \citep[Lemma 2.2]{dk2013combinatorics} that each $\Omega^t(\bar{\bm{w}})$ is nonempty. Hence, $\mathcal{C}_{\mathrm{feas}}$ is nonempty. Moreover, it is not hard to see that $\mathcal{C}_{\mathrm{feas}}$ is closed and bounded. This together with the continuity of the objective function in \eqref{subpro1} implies that the optimal solution of \eqref{subpro1} exists. Hence, the optimal solution of \eqref{repro} exists. Now, let $(\bm{w}^*,\,\{\Pi^{(t),*}\})$ be an optimal solution of \eqref{repro}. Since the set $\{(\bm{w},\,\{\Pi^{(t)}\}): \bm{w}\in\Delta_{m}, \,\Pi^{(t)}\geq0,\,t=1,\cdots,N\}$ is a convex polyhedron and all constraint functions in \eqref{repro} are affine, then it follows from \citep[Theorem 3.25]{r2006nonlinear} that there exist multipliers $\bm{y}^{(t),*} \in \mathbb{R}^m$, $\bm{z}^{(t),*}\in\mathbb{R}^{m_t}$, $t = 1, \cdots, N$ such that $\left(\bm{w}^*, \{\Pi^{(t),*}\}, \{\bm{y}^{(t),*}\}, \{\bm{z}^{(t),*}\}\right)$ satisfies the KKT system \eqref{KKTprimal}. Thus, the solution set of the KKT system \eqref{KKTprimal} is also nonempty. This proves statement (i).

\textit{Statement} (ii). Let $\left(\bm{w}^*, \{\Pi^{(t),*}\}, \{\bm{y}^{(t),*}\}, \{\bm{z}^{(t),*}\}\right)$ be a solution of the KKT system \eqref{KKTprimal}. It follows from statement (i) that such a solution exists. Now, consider $\bm{u}^* = {\textstyle\sum^N_{t=1}}\bm{y}^{(t),*}$, $\bm{\lambda}^*=\bm{w}^*$, $\Lambda^{(t),*}=\Pi^{(t),*}$, $V^{(t),*} = D^{(t)} + \bm{y}^{(t),*}\bm{e}^{\top}_{m_t} + \bm{e}_{m}(\bm{z}^{(t),*})^{\top}$, $t=1,\cdots,N$. Then, by simple calculations and recalling \eqref{subeqv}, one can verify that $\big{(}\bm{u}^*, \{V^{(t),*}\}, \{\bm{y}^{(t),*}\}, \{\bm{z}^{(t),*}\}, \bm{\lambda}^*$, $\{\Lambda^{(t),*}\}\big{)}$ satisfies the KKT system \eqref{KKTdual}. Hence, the solution set of the KKT system \eqref{KKTdual} is nonempty. Moreover, from \citep[Theorem 3.27]{r2006nonlinear}, we see that $\big{(}\bm{u}^*, \{V^{(t),*}\}$, $\{\bm{y}^{(t),*}\}, \{\bm{z}^{(t),*}\}\big{)}$ is also an optimal solution of \eqref{subdualpro}. This shows that the optimal solution of \eqref{subdualpro} exists.

\textit{Statement} (iii). First, it is easy to see from
\citep[Theorem 3.27]{r2006nonlinear} that $(\bm{u}^*$, $\{V^{(t),*}\}$, $\{\bm{y}^{(t),*}\}$, $\{\bm{z}^{(t),*}\})$ solves problem \eqref{subdualpro}. Then, simplifying the KKT system \eqref{KKTdual} and recalling \eqref{subeqv}, one can verify that $(\bm{\lambda}^*$, $\{\Lambda^{(t),*}\}$, $\{\bm{y}^{(t),*}\}$, $\{\bm{z}^{(t),*}\})$ satisfies the KKT system \eqref{KKTprimal} with $\bm{\lambda}^*$ in place of $\bm{w}$ and $\Lambda^{(t),*}$ in place of $\Pi^{(t)}$. Now, using \citep[Theorem 3.27]{r2006nonlinear} again, we see that $\left(\bm{\lambda}^*, \{\Lambda^{(t),*}\}\right)$ is an optimal solution of \eqref{repro}. This proves statement (iii).
\end{proof}

In order to present the global convergence of Algorithm \ref{sGSADMMdual} based on the theory developed in \citep{fpst2013hankel,lst2016schur}, we first express problem \eqref{subdualpro} as follows:
\begin{equation*}
\begin{aligned}
&\min\limits_{\bm{u}, \{V^{(t)}\}, \{\bm{y}^{(t)}\}, \{\bm{z}^{(t)}\}} ~\theta\big{(}\bm{u}, \{V^{(t)}\}\big{)} + g\big{(}\{\bm{y}^{(t)}\}, \{\bm{z}^{(t)}\}\big{)} \vspace{-2mm} \\
&\hspace{1.3cm}\mathrm{s.t.}\hspace{1.3cm}
A\begin{bmatrix*}[c]
\bm{u} \\ \mathrm{vec}(V^{(1)}) \vspace{-1.5mm} \\ \vdots \\ \mathrm{vec}(V^{(N)})
\end{bmatrix*}
+ B_1
\begin{bmatrix*}[c]
\bm{y}^{(1)}  \vspace{-1.5mm} \\ \vdots \\ \bm{y}^{(N)}
\end{bmatrix*}
+ B_2
\begin{bmatrix*}[c]
\bm{z}^{(1)}  \vspace{-1.5mm} \\ \vdots \\ \bm{z}^{(N)}
\end{bmatrix*}
=
\begin{bmatrix*}[c]
0 \\ \mathrm{vec}(D^{(1)}) \vspace{-1.5mm} \\ \vdots \\ \mathrm{vec}(D^{(N)})
\end{bmatrix*},
\end{aligned}
\end{equation*}
where $\theta\big{(}\bm{u}, \{V^{(t)}\}\big{)} = \delta^*_{\Delta_m}(\bm{u}) + {\textstyle\sum^N_{t=1}} \delta_{+}^t(V^{(t)})$,
$g\big{(}\{\bm{y}^{(t)}\},\{\bm{z}^{(t)}\}\big{)}={\textstyle\sum^N_{t=1}}\langle \bm{z}^{(t)}, \,\bm{a}^{(t)}\rangle$ and
\begin{equation}\label{defmatrix}
A =
\begin{bmatrix*}[c]
-I_m &\!\!               \\
     &\!\! I_{m\sum_{t} m_t} \\
\end{bmatrix*},
\;\;
B_1 =
\begin{bmatrix*}[c]
1 &\!\!\!\! \cdots  &\!\!\!\!\! 1  \\
-\bm{e}_{m_1} &\!\!\!\!  &\!\!\!\!\!   \\
&\!\!\!\! \ddots  &\!\!\!\!\!  \vdots        \\
&\!\!\!\!   &\!\!\!\!\! -\bm{e}_{m_N}
\end{bmatrix*}\otimes I_m,
\;\;
B_2 =
\begin{bmatrix*}[c]
0 &\!\!\!\! \cdots&\!\!\!\!\! 0         \\
-I_{m_1}&\!\!\!\!  &\!\!\!\!\!       \\
&\!\!\!\!\ddots &\!\!\!\!\!  \vdots     \\
&\!\!\!\!   &\!\!\!\!\! -I_{m_N}
\end{bmatrix*} \otimes \bm{e}_{m}.
\end{equation}
It is easy to verify that $A^{\top}A = I_{m\left(1+\sum_{t} m_t\right)} \succ 0$ and
\begin{equation*}
B_1^{\top}B_1 =
\begin{bmatrix*}[c]
m_1+1 &\!\!\!         &        \\
      &\!\!\!  \ddots &        \\
      &\!\!\!         & m_N+1
\end{bmatrix*} \otimes I_m \succ 0, \quad
B_2^{\top}B_2 = m
\begin{bmatrix*}[c]
 I_{m_1} &\!\!\!         &        \\
         &\!\!\!  \ddots &        \\
         &\!\!\!         & I_{m_N}
\end{bmatrix*} \succ 0.
\quad
\end{equation*}

For notational simplicity, denote
\begin{equation*}
\begin{aligned}
&\mathcal{W}:=(\bm{u}, \{V^{(t)}\}, \{\bm{y}^{(t)}\}, \{\bm{z}^{(t)}\}, \bm{\lambda}, \{\Lambda^{(t)}\}), \\
&\mathcal{W}^k:=(\bm{u}^k, \{V^{(t),k}\}, \{\bm{y}^{(t),k}\}, \{\bm{z}^{(t),k}\}, \bm{\lambda}^k, \{\Lambda^{(t),k}\}), \\
&\bm{y}:= [\bm{y}^{(1)};\cdots;\bm{y}^{(N)}],
~~\bm{y}^k:= [\bm{y}^{(1),k};\cdots;\bm{y}^{(N),k}], \\
&\bm{z}:=[\bm{z}^{(1)};\cdots;\bm{z}^{(N)}],
~~\bm{z}^k:=[\bm{z}^{(1),k};\cdots;\bm{z}^{(N),k}], \\
&\bm{v}:=[\mathrm{vec}(V^{(1)});\cdots;\mathrm{vec}(V^{(N)})],
~~\bm{v}^k:=[\mathrm{vec}(V^{(1),k});\cdots;\mathrm{vec}(V^{(N),k})], \\
&\bm{d}:= [0;  \mathrm{vec}(D^{(1)}); \cdots ; \mathrm{vec}(D^{(N)})],
~~\mathrm{vec}(\{\Lambda^{(t)}\}):=[\mathrm{vec}(\Lambda^{(1)});\cdots;\mathrm{vec}(\Lambda^{(N)})], \\
&\mathrm{vec}(\mathcal{W}):=[\bm{u};\bm{v};\bm{y};\bm{z};\bm{\lambda};\mathrm{vec}(\{\Lambda^{(t)}\})].
\end{aligned}
\end{equation*}
By using the above notation, we can rewrite problem \eqref{subdualpro} in a compact form as follows:
\begin{equation}\label{eq-compact}
\begin{aligned}
&\min~~\theta(\bm{u}, \bm{v}) + g(\bm{y},\bm{z}) \\
&~\mbox{s.t.} ~~~ A [\bm{u}; \bm{v}] + B [\bm{y}; \bm{z}] = \bm{d},
\end{aligned}
\end{equation}
where $B = [B_1~B_2]$. Then, our sGS-ADMM (Algorithm \ref{sGSADMMdual}) is precisely a 2-block sPADMM applied to the compact form \eqref{eq-compact} of \eqref{subdualpro} with a specially designed proximal term.
In particular, {\bf Step 1} of the algorithm is the same as computing
\begin{equation}\label{eq-step11}
(\bm{u}^{k+1},\,\bm{v}^{k+1})
= \arg\min_{\bm{u},\bm{v}} \;\big\{\, \mathcal{L}_\beta(\bm{u},\bm{v},\bm{y}^k,\bm{z}^k;\bm{\lambda}^k,
\{\Lambda^{(t),k}\}) \,\big\}.
\end{equation}
It follows from \citep[Proposition 5]{lst2016schur} that {\bf Step 2a--2c} is equivalent to
\begin{equation}\label{eq-step22}
(\bm{y}^{k+1},\,\bm{z}^{k+1}) = \arg\min\limits_{\bm{y},\bm{z}}\left\{
\mathcal{L}_\beta(\bm{u}^{k+1}, \bm{v}^{k+1}, \bm{y}, \bm{z}; \bm{\lambda}^k,
\{\Lambda^{(t),k}\}) + \textstyle\frac{\beta}{2}\| [\bm{y}; \bm{z}]- [\bm{y}^k; \bm{z}^k] \|^2_C\right\},
\end{equation}
where the matrix $C$ in the proximal term is the symmetric Gauss-Seidel decomposition operator of $B^{\top} B$ and it is given by
\begin{equation*}
C=
\begin{bmatrix*}[c]
B_1^{\top}B_2\left(B_2^{\top}B_2\right)^{-1}B_2^{\top}B_1
& 0  \\
0 & 0
\end{bmatrix*}.
\end{equation*}
Based on the above fact that the sGS-ADMM can be reformulated as a 2-block sPADMM with a specially designed semi-proximal term, one can directly obtain the global convergence of Algorithm \ref{sGSADMMdual} from that of the 2-block sPADMM.

\begin{theorem}\label{conthm}
Let $\beta>0$, $\tau\in(0, \frac{1+\sqrt{5}}{2})$ and $\big{\{}\big{(}\,\bm{u}^k, \{V^{(t),k}\}, \{\bm{y}^{(t),k}\}, \{\bm{z}^{(t),k}\}, \bm{\lambda}^k, \{\Lambda^{(t),k}\}\,\big{)}\big{\}}$ be the sequence generated by the sGS-ADMM in Algorithm \ref{sGSADMMdual}. Then, the sequence $\big{\{}\big{(}\,\bm{u}^k, \{V^{(t),k}\}$, $\{\bm{y}^{(t),k}\}, \{\bm{z}^{(t),k}\}\,\big{)}\big{\}}$ converges to an optimal solution of \eqref{subdualpro} and the sequence $\big{\{}\big{(}\,\bm{\lambda}^k, \{\Lambda^{(t),k}\}\,\big{)}\big{\}}$ converges to an optimal solution of \eqref{repro}.
\end{theorem}
\begin{proof}
Here we apply the convergence result developed in \citep{fpst2013hankel} to the 2-block sPADMM outlined in \eqref{eq-step11}, \eqref{eq-step22} and {\bf Step 3} of Algorithm \ref{sGSADMMdual}. Since both $A^\top A$ and $\beta C + \beta B^\top B$ are positive definite, the conditions for ensuring the convergence of the 2-block sPADMM in \citep[Theorem B.1]{fpst2013hankel} are satisfied, thus along with Proposition \ref{existopt}, one can readily apply
\citep[Theorem B.1]{fpst2013hankel} to obtain the desired results.
\end{proof}

Moreover, based on the equivalence of our sGS-ADMM to a 2-block sPADMM, the linear convergence rate of the sGS-ADMM can also be established from the linear convergence result of the 2-block sPADMM; see \citep[Section 4.1]{hsz2017linear} for more details.

Define
\begin{equation*}
\begin{aligned}
&M
:=\begin{bmatrix*}[c]
0 & & \\
& \beta C+\beta B^{\top}B & \\
&& (\tau\beta)^{-1}I_{m\left(1+\sum^N_{t=1} m_t\right)}
\end{bmatrix*} + s_{\tau}\beta
\begin{bmatrix*}[c]
A^{\top}A & A^{\top} B & 0 \\
B^{\top}A & B^{\top}B& 0 \\
0 & 0 & 0
\end{bmatrix*},
\end{aligned}
\end{equation*}
where $A$, $B_1$, $B_2$ are defined in \eqref{defmatrix} and $s_{\tau}:=(5-\tau-3\min\{\tau,\,\tau^{-1}\})/4$. One can verify that $M\succ0$. Indeed, it is easy to see from the definition that $M\succ0$ if and only if
\begin{equation*}
M_1:=
\begin{bmatrix*}[c]
A^{\top}A & A^{\top}B \\
B^{\top}A & s_{\tau}^{-1}C + (1+s_{\tau}^{-1}) B^{\top}B
\end{bmatrix*} \succ 0.
\end{equation*}
Thus, one only needs to verify that $M_1\succ0$. Note that $A^{\top}A = AA^{\top} = I_{m\left(1+\sum^N_{t=1} m_t\right)} \succ 0$. The Schur complement of $A^{\top}A$ takes the form of
\begin{equation*}
\begin{aligned}
M_2
~:=&~s_{\tau}^{-1} C + (1+s_{\tau}^{-1}) B^{\top}B - B^{\top}A(A^{\top}A)^{-1}A^{\top}B
\;\;=~s_{\tau}^{-1}C + s_{\tau}^{-1} B^{\top}B
 \\
=&~s_{\tau}^{-1}
\begin{bmatrix*}[c]
B_1^{\top}B_2\big{(}B_2^{\top}B_2\big{)}^{-1}B_2^{\top}B_1 + B_1^{\top}B_1 & B_1^{\top}B_2  \\
B_2^{\top}B_1 & B_2^{\top}B_2
\end{bmatrix*}.
\end{aligned}
\end{equation*}
Since $B_2^{\top}B_2\succ0$ and its Schur complement satisfies
\begin{equation*}
B_1^{\top}B_2\big{(}B_2^{\top}B_2\big{)}^{-1}B_2^{\top}B_1 + B_1^{\top}B_1 - B_1^{\top}B_2\big{(}B_2^{\top}B_2\big{)}^{-1}B_2^{\top}B_1 = B_1^{\top}B_1 \succ 0,
\end{equation*}
then $M_2\succ0$. This implies that $M_1\succ0$ and hence $M\succ0$.

We also let $\mathscr{W} := \mathbb{R}^m \times \otimes_{t=1}^N\mathbb{R}^{m \times m_t} \times \mathbb{R}^m \times \otimes_{t=1}^N\mathbb{R}^{m_t} \times \mathbb{R}^m \times \otimes_{t=1}^N\mathbb{R}^{m \times m_t}$ and $\Omega\subseteq\mathscr{W}$ be the solution set of the KKT system \eqref{KKTdual}. Recall from Proposition \ref{existopt}(ii) that $\Omega$ is nonempty. Moreover, for any $\mathcal{W}\in\mathscr{W}$, we define
\begin{equation*}
\begin{aligned}
\mathrm{dist}(\mathcal{W}, \,\Omega) &:= \inf\limits_{\mathcal{W}'\in\Omega} \|\mathrm{vec}(\mathcal{W})-\mathrm{vec}(\mathcal{W}')\|, \\
\mathrm{dist}_M(\mathcal{W}, \,\Omega) &:= \inf\limits_{\mathcal{W}'\in\Omega} \|\mathrm{vec}(\mathcal{W})-\mathrm{vec}(\mathcal{W}')\|_M.
\end{aligned}
\end{equation*}
Since $M\succ0$, $\mathrm{dist}_M$ is also a point-to-set distance. We  present the linear convergence result of our sGS-ADMM in the next theorem.

\begin{theorem}\label{thmconrate}
Let $\beta>0$, $\tau\in(0, \frac{1+\sqrt{5}}{2})$ and $\left\{\mathcal{W}^k\right\}$ be the sequence generated by the sGS-ADMM in Algorithm \ref{sGSADMMdual}. Then, there exists a constant $0<\rho<1$ such that, for all $k\geq1$,
\begin{equation*}
\mathrm{dist}_M^2(\mathcal{W}^{k+1}, \,\Omega) + \beta\big{\|}[\bm{y}^{k+1}; \bm{z}^{k+1}] - [\bm{y}^k; \bm{z}^k]\big{\|}^2_{C}
\leq \rho\Big{(} \mathrm{dist}_M^2(\mathcal{W}^{k}, \,\Omega) + \beta\big{\|}[\bm{y}^{k}; \bm{z}^{k}] - [\bm{y}^{k-1}; \bm{z}^{k-1}]\big{\|}^2_{C} \Big{)}.
\end{equation*}
\end{theorem}
\begin{proof}
First we note the equivalence of the sGS-ADMM to a 2-block sPADMM. Next consider the KKT mapping $\mathcal{R}: \mathscr{W} \to \mathscr{W}$ defined by
\begin{equation*}
\mathcal{R}(\mathcal{W})
:=
\left(\begin{array}{c}
\bm{\lambda} - \mathrm{Pr}_{\Delta_m}(\bm{\lambda}+\bm{u}) \\
\left\{V^{(t)} - \mathrm{Pr}_{+}^t (V^{(t)}-\Lambda^{(t)})\right\}   \\
\left\{\Lambda^{(t)}\bm{e}_{m_t} - \bm{\lambda}\right\}   \\
\left\{(\Lambda^{(t)})^{\top}\bm{e}_m - \bm{a}^{(t)}\right\} \\
{\textstyle\sum^N_{t=1}}\bm{y}^{(t)} - \bm{u}  \\
\left\{V^{(t)} - D^{(t)} - \bm{y}^{(t)}\bm{e}^{\top}_{m_t} - \bm{e}_{m}(\bm{z}^{(t)})^{\top}\right\}
\end{array}\right), \quad \forall\,\mathcal{W} \in \mathscr{W},
\end{equation*}
where $\mathrm{Pr}_{\Delta_m}(\cdot)$ denotes the projection operator over $\Delta_m$ and $\mathrm{Pr}_{+}^t(\cdot)$ denotes the projection operator over $\mathbb{R}^{m\times m_t}_+$ for $t=1,\cdots,N$. It is easy to see that $\mathcal{R}(\cdot)$ is continuous on $\mathscr{W}$. Moreover, note that $\bm{\lambda}\in\partial \delta^*_{\Delta_m}(\bm{u})\Longleftrightarrow
\bm{u}\in\partial\delta_{\Delta_m}(\bm{\lambda})\Longleftrightarrow
0\in\partial\delta_{\Delta_m}(\bm{\lambda})+\bm{\lambda}-(\bm{\lambda}+\bm{u})
\Longleftrightarrow \bm{\lambda}=\mathrm{Prox}_{\delta_{\Delta_m}}(\bm{\lambda}+\bm{u})
=\mathrm{Pr}_{\Delta_m}(\bm{\lambda}+\bm{u})$, where the first equivalence follows from \eqref{subeqv}. Similarly, $-\Lambda^{(t)}\in\partial \delta_{+}^t (V^{(t)}) \Longleftrightarrow V^{(t)}=(V^{(t)}-\Lambda^{(t)})_+$, where $(\cdot)_+=\max(\cdot\,, 0)$. Using these facts, one can easily see that $\mathcal{R}(\mathcal{W})=0$ if and only if $\mathcal{W} \in \Omega$. By Theorem \ref{conthm}, we know that the sequence $\{ \mathcal{W}^k\}$ converges to an optimal solution $\mathcal{W}^*\in \Omega$, and hence $\mathcal{R}(\mathcal{W}^*)=0.$

Now, since $\Delta_m$, $\mathbb{R}^{m\times m_1}_+, \cdots, \mathbb{R}^{m\times m_N}_+$ are polyhedral, it follows from \citep[Example 11.18]{rw1998variational} and the definition of projections that $\mathrm{Pr}_{\Delta_m}(\cdot)$ and $\mathrm{Pr}_{+}^t(\cdot)$ are piecewise polyhedral. Hence, $\mathcal{R}(\cdot)$ is also piecewise polyhedral.
From \citep{Robinson}, we know that the KKT mapping $\mathcal{R}$ satisfies the following error bound condition:
there exist two positive scalars  $\eta>0$ and  $\tilde{\rho}>0$ such that
\begin{equation*}
{\rm dist}(\mathcal{W},\,\Omega) \;\leq\; \eta\|\mathrm{vec}({\mathcal{R}(\mathcal{W})})\|, \quad \forall\;
\mathcal{W} \in \{ \mathcal{W} \mid \|\mathrm{vec}({\mathcal{R}(\mathcal{W})})\| \leq \tilde{\rho}\},
\end{equation*}
where $\mathrm{vec}({\mathcal{R}(\mathcal{W})})$ denotes the vectorization of $\mathcal{R}(\mathcal{W})$.

Finally, based on the above facts and Proposition \ref{existopt}, we can apply \citep[Corollary 1]{hsz2017linear} to obtain the desired results.
\end{proof}

\section{Numerical Experiments}\label{sectnum}

In this section, we conduct numerical experiments to test our sGS-ADMM in Algorithm \ref{sGSADMMdual} for computing Wasserstein barycenters with pre-specified support points, i.e., solving problem \eqref{subpro1}. In all our experiments, we use the 2-Wasserstein distance. We also compare our sGS-ADMM with the commercial software Gurobi and two existing representative methods, namely, the iterative Bregman projection (IBP) method \citep{bccnp2015iterative} and the modified Bregman ADMM (BADMM) \citep{ywwl2017fast}. For ease of future reference, we briefly recall IBP and BADMM in Appendices \ref{apdIBP} and \ref{apdbadmm}, respectively. All experiments are run in {\sc Matlab} R2016a on a workstation with Intel(R) Xeon(R) Processor E-2176G@3.70GHz (this processor has 6 cores and 12 threads) and 64GB of RAM, equipped with 64-bit Windows 10 OS.

\subsection{Implementation Details}\label{secImple}

In our implementation of the sGS-ADMM, a data scaling technique is used. Let $\kappa=\|[D^{(1)}, \cdots, D^{(N)}]\|_F$. Then, problem \eqref{subpro1} is equivalent to
\begin{equation}\label{scalpro}
\begin{aligned}
&\min\limits_{\bm{w},\,\{\Pi^{(t)}\}}~{\textstyle\sum^N_{t=1}}\langle \widehat{D}^{(t)}, \,\Pi^{(t)} \rangle  \\
&\hspace{0.5cm}\mathrm{s.t.} \hspace{0.5cm} \Pi^{(t)}\bm{e}_{m_t} = \bm{w}, ~(\Pi^{(t)})^{\top}\bm{e}_{m} = \bm{a}^{(t)},~\Pi^{(t)} \geq 0, ~~\forall\, t = 1, \cdots, N, \\
&\hspace{1.5cm} \bm{e}^{\top}_m\bm{w} = 1, ~\bm{w} \geq 0,
\end{aligned}
\end{equation}
where $\widehat{D}^{(t)} = \kappa^{-1} D^{(t)}$ for $t=1,\cdots,N$. We then apply the sGS-ADMM to solve the dual problem of \eqref{scalpro} to obtain an optimal solution of \eqref{subpro1}. Indeed, this technique has been widely used in ADMM-based methods to improve their numerical performances; see, for example, \citep{lmst2017fast}. Its effectiveness has also been observed in our experiments.

For a set of vectors $\{ \bm{a}^{(t)} \!\mid\! t=1,\!\cdots\!,N\}$, we define the notation $\|\{ \bm{a}^{(t)}\}\| :=
\big{(}\sum_{t=1}^N \| \bm{a}^{(t)}\|^2\big{)}^{\frac{1}{2}}$. Similarly, for a set of matrices $\{ A^{(t)} \mid t=1,\ldots, N\}$, we define the notation $\|\{ A^{(t)}\}\|_F := \big{(}\sum_{t=1}^N \| A^{(t)}\|_F^2\big{)}^{\frac{1}{2}}$. For any $\bm{u}, \{V^{(t)}\}, \{\bm{y}^{(t)}\}, \{\bm{z}^{(t)}\}, \bm{\lambda}, \{\Lambda^{(t)}\}$, we define the relative residuals based on the KKT system \eqref{KKTdual} as follows:
\begin{equation*}
\begin{array}{ll}
~~\eta_1(\bm{\lambda},\bm{u}) = {\textstyle\frac{\left\|\bm{\lambda}-\mathrm{Pr}_{\Delta_m}(\bm{\lambda}+\bm{u})\right\|}{1+\|\bm{\lambda}\|+\|\bm{u}\|}},
&\eta_2(\{V^{(t)}\},\{\Lambda^{(t)}\}) =
{\textstyle\frac{\|\{V^{(t)}-(V^{(t)}-\Lambda^{(t)})_+\}\|_F}{1+\|\{V^{(t)}\}\|_F+\|\{\Lambda^{(t)}\}\|_F}}, \\[8pt]
~~\eta_3(\bm{\lambda},\{\Lambda^{(t)}\}) = {\textstyle\frac{\|\{\Lambda^{(t)}\bm{e}_{m_t}-\bm{\lambda}\}\|}{1+\|\bm{\lambda}\|+\|\{\Lambda^{(t)}\}\|_F}},
&\eta_4(\{\Lambda^{(t)}\}) = {\textstyle\frac{\|\{(\Lambda^{(t)})^{\top}\bm{e}_m - \bm{a}^{(t)}\}\|_F}{1+\|\{\bm{a}^{(t)}\}\|+\|\{\Lambda^{(t)}\}\|_F}},
\\[8pt]
~~\eta_5(\bm{u},\{\bm{y}^{(t)}\}) = {\textstyle\frac{\|\sum^N_{t=1}\bm{y}^{(t)} - \bm{u}\|}{1+\|\sum^N_{t=1}\bm{y}^{(t)}\|+\|\bm{u}\|}},
&\eta_6(\{V^{(t)}\}, \{\bm{y}^{(t)}\}, \{\bm{z}^{(t)}\}) =
{\textstyle\frac{\|\{V^{(t)}-D^{(t)}-\bm{y}^{(t)}\bm{e}^{\top}_{m_t}-\bm{e}_{m}(\bm{z}^{(t)})^{\top}\}\|_F}{1+\|\{D^{(t)}\}\|_F
+\|\{V^{(t)}\}\|_F+\|\{\bm{y}^{(t)}\}\|+\|\{\bm{z}^{(t)}\}\|} }, \\[8pt]
~~\eta_7(\bm{\lambda}) = {\textstyle\frac{|\bm{e}_m^{\top}\bm{\lambda}-1|+\|\mathrm{min}(\bm{\lambda},\,0)\|}{1+\|\bm{\lambda}\|}},
&\eta_8(\{\Lambda^{(t)}\}) = {\textstyle\frac{\|\mathrm{min}([\Lambda^{(1)}, \cdots, \Lambda^{(N)}], \,0)\|_F}{1+\|\{\Lambda^{(t)}\}\|_F}}.
\end{array}
\end{equation*}
Moreover, let $\mathcal{W}=(\bm{u}, \{V^{(t)}\}, \{\bm{y}^{(t)}\}, \{\bm{z}^{(t)}\}, \bm{\lambda}, \{\Lambda^{(t)}\})$ and
\begin{equation*}
\begin{aligned}
\eta_P(\mathcal{W}) &= \max\big{\{}\eta_1(\bm{\lambda},\bm{u}),\,0.7\eta_2\big{(}\{V^{(t)}\},\{\Lambda^{(t)}\}\big{)},
\,\eta_3\big{(}\bm{\lambda},\{\Lambda^{(t)}\}\big{)},\,\eta_4\big{(}\{\Lambda^{(t)}\}\big{)}\big{\}}, \\
\eta_D(\mathcal{W}) &= \max\big{\{}0.7\eta_5\big{(}\bm{u},\{\bm{y}^{(t)}\}\big{)},\,\eta_6\big{(}\{V^{(t)}\}, \{\bm{y}^{(t)}\}, \{\bm{z}^{(t)}\}\big{)},\,\eta_7(\bm{\lambda}),\,0.7\eta_8\big{(}\{\Lambda^{(t)}\}\big{)}\big{\}}.
\end{aligned}
\end{equation*}
Following discussions in Theorem \ref{thmconrate}, it is easy to verify that $\max\{\eta_P(\mathcal{W}), \eta_D(\mathcal{W})\}=0$ if and only if $\mathcal{W}$ is a solution of the KKT system \eqref{KKTdual}. The relative duality gap is defined by
\begin{equation*}
\eta_{gap}(\mathcal{W}) := \frac{|\,\mathrm{obj}_{P}(\mathcal{W})-\mathrm{obj}_{D}(\mathcal{W})\,|}
{1+|\,\mathrm{obj}_{P}(\mathcal{W})\,|+|\,\mathrm{obj}_{D}(\mathcal{W})\,|},
\end{equation*}
where $\mathrm{obj}_{P}(\mathcal{W})={\textstyle\sum^N_{t=1}}\langle D^{(t)}, \,\Pi^{(t)} \rangle$ and $\mathrm{obj}_{D}(\mathcal{W})=\delta^*_{\Delta_m}\big{(}{\textstyle\sum^N_{t=1}\bm{y}^{(t)}}\big{)} + {\textstyle\sum^N_{t=1}}\langle \bm{z}^{(t)}, \,\bm{a}^{(t)}\rangle$. We use these relative residuals in our stopping criterion for the sGS-ADMM. Specifically, we will terminate the sGS-ADMM when
\begin{equation*}
\max\big{\{}\eta_P(\mathcal{W}^{k+1}), \,\eta_D(\mathcal{W}^{k+1}), \,\eta_{gap}(\mathcal{W}^{k+1})\big{\}} < \mathrm{Tol}_{\mathrm{sgs}},
\end{equation*}
where $\mathcal{W}^{k+1}$ is generated by the sGS-ADMM at the $k$-th iteration and the value of $\mathrm{Tol}_{\mathrm{sgs}}$ will be given later.

We also use a similar numerical strategy as \citep[Section 4.4]{lmst2017fast} to update the penalty parameter $\beta$ in the augmented Lagrangian function at every 50 iterations. Specifically, set $\beta_0=1$. At the $k$-th iteration, if $\mathrm{mod}(k,\,50)\neq0$, set $\beta_{k+1}=\beta_k$; otherwise, compute $\chi^{k+1}=\frac{\eta_D(\mathcal{W}^{k+1})}{\eta_P(\mathcal{W}^{k+1})}$ and then, set
\begin{equation*}
\beta_{k+1} = \left\{
\begin{array}{ll}
\sigma\beta_k,        & \mathrm{if}~~\chi^{k+1}>2, \\
\sigma^{-1}\beta_k,   & \mathrm{if}~~\frac{1}{\chi^{k+1}}>2, \\
\beta_k,              & \mathrm{otherwise},
\end{array}\right.
~~ \mathrm{with} ~~
\sigma = \left\{
\begin{array}{ll}
1.1,   & \mathrm{if}~~\max\{\chi^{k+1},\,\frac{1}{\chi^{k+1}}\} \leq 50, \\
2,     & \mathrm{if}~~ \max\{\chi^{k+1},\,\frac{1}{\chi^{k+1}}\} > 500, \\
1.5,   & \mathrm{otherwise},
\end{array}\right.
\end{equation*}
where $\mathrm{mod}(k,\,50)$ denotes the remainder after division of $k$ by $50$. Note that the value of $\beta$ is adjusted based on the primal and dual information. As observed from our experiments, this updating strategy can efficiently balance the convergence of the primal and dual variables, and improve the convergence speed of our algorithm.

Computing all the above residuals is expensive. Thus, in our implementations, we only compute them and check the termination criteria at every 50 iterations. In addition, we initialize the sGS-ADMM at origin and choose the dual step-size $\tau$ to be 1.618.

For IBP, the regularization parameter $\varepsilon$ is chosen from $\{0.1, 0.01, 0.001\}$ in our experiments. For $\varepsilon\in\{0.1,0.01\}$, we follow \citep[Remark 3]{bccnp2015iterative} to implement the algorithm (see \eqref{iterschemeIBP}) and terminate it when
\begin{equation*}
{\textstyle\frac{\|\bm{w}^{k+1}-\bm{w}^k\|}{1+\|\bm{w}^{k+1}\|+\|\bm{w}^k\|} < \mathrm{Tol}_{\mathrm{ibp}}}, ~
{\textstyle\frac{\|\{\bm{u}^{(t),k+1}-\bm{u}^{(t),k}\}\|}{1+\|\{\bm{u}^{(t),k+1}\}\|+\|\{\bm{u}^{(t),k}\}\|} < \mathrm{Tol}_{\mathrm{ibp}}}, ~
{\textstyle\frac{\|\{\bm{v}^{(t),k+1}-\bm{v}^{(t),k}\}\|_F}{1+\|\{\bm{v}^{(t),k+1}\}\|+\|\{\bm{v}^{(t),k}\}\|} < \mathrm{Tol}_{\mathrm{ibp}}},
\end{equation*}
where $(\bm{w}^{k+1}, \{\bm{u}^{(t),k+1}\}, \{\bm{v}^{(t),k+1}\})$ is generated at the $k$-th iteration in \eqref{iterschemeIBP}. Moreover, for $\varepsilon=0.001$, we follow \citep[Section 4.4]{pc2019computational} to adapt the \textit{log-sum-exp} trick for stabilizing IBP (see \eqref{iterschemeIBPlog}). This stabilized IBP is terminated when
\begin{equation*}
{\textstyle\frac{\|\tilde{\bm{w}}^{k+1}-\tilde{\bm{w}}^k\|}{1+\|\tilde{\bm{w}}^{k+1}\|+\|\tilde{\bm{w}}^k\|} < \mathrm{Tol}_{\mathrm{ibp}}}, ~
{\textstyle\frac{\|\{\tilde{\bm{u}}^{(t),k+1}-\tilde{\bm{u}}^{(t),k}\}\|}{1+\|\{\tilde{\bm{u}}^{(t),k+1}\}\|+\|\{\tilde{\bm{u}}^{(t),k}\}\|} < \mathrm{Tol}_{\mathrm{ibp}}}, ~
{\textstyle\frac{\|\{\tilde{\bm{v}}^{(t),k+1}-\tilde{\bm{v}}^{(t),k}\}\|_F}{1+\|\{\tilde{\bm{v}}^{(t),k+1}\}\|+\|\{\tilde{\bm{v}}^{(t),k}\}\|} < \mathrm{Tol}_{\mathrm{ibp}}},
\end{equation*}
where $(\tilde{\bm{w}}^{k+1}, \{\tilde{\bm{u}}^{(t),k+1}\}, \{\tilde{\bm{v}}^{(t),k+1}\})$ is generated at the $k$-th iteration in \eqref{iterschemeIBPlog}. The value of $\mathrm{Tol}_{\mathrm{ibp}}$ will be given later.

For BADMM, we use the Matlab codes\footnote{Available in \url{https://github.com/bobye/WBC_Matlab}.} implemented by the authors in \citep{ywwl2017fast} and terminate them when
\begin{equation*}
\begin{array}{ll}
\max\big{\{}\eta_3\big{(}\bm{w}^{k+1}, \{\Gamma^{(t),k+1}\}\big{)}, \,\eta_4(\{\Pi^{(t),k+1}\})\big{\}} < \mathrm{Tol}_{\mathrm{b}},
&\frac{\|\bm{w}^{k+1}-\bm{w}^k\|}{1+\|\bm{w}^{k+1}\|+\|\bm{w}^k\|} < \mathrm{Tol}_{\mathrm{b}}, \\[8pt]
{\frac{\|\{\Pi^{(t),k+1}-\Gamma^{(t),k+1}\}\|_F}{1+\|\{\Pi^{(t),k+1}\}\|_F+\|\{\Gamma^{(t),k+1}\}\|_F} < \mathrm{Tol}_{\mathrm{b}}},
&{\frac{\|\{\Pi^{(t),k+1}-\Pi^{(t),k}\}\|_F}{1+\|\{\Pi^{(t),k}\}\|_F+\|\{\Pi^{(t),k+1}\}\|_F} < \mathrm{Tol}_{\mathrm{b}}}, \\[8pt]
{\frac{\|\{\Gamma^{(t),k+1}-\Gamma^{(t),k}\}\|_F}{1+\|\{\Gamma^{(t),k}\}\|_F+\|\{\Gamma^{(t),k+1}\}\|_F} < \mathrm{Tol}_{\mathrm{b}}},
&{\frac{\|\{\Lambda^{(t),k+1}-\Lambda^{(t),k}\}\|_F}{1+\|\{\Lambda^{(t),k}\}\|_F+\|\{\Lambda^{(t),k+1}\}\|_F} < \mathrm{Tol}_{\mathrm{b}}},
\end{array}
\end{equation*}
where $(\bm{w}^{k+1}, \{\Pi^{(t),k+1}\}, \{\Gamma^{(t),k+1}\}, \{\Lambda^{(t),k+1}\})$ is generated by BADMM at the $k$-th iteration (see Appendix \ref{apdbadmm}) and the value of $\mathrm{Tol}_{\mathrm{b}}$ will be given later. The above termination criteria are checked at every 200 iterations.

We also apply Gurobi 8.0.0 \citep{gurobi} to solve problem \eqref{subpro1}. It is well known that Gurobi is a highly powerful commercial package for solving linear programming problems and can provide high quality solutions. Therefore, we use the solution obtained by Gurobi as a benchmark to evaluate the qualities of solutions obtained by different methods. In our experiments, we use the default parameter settings for Gurobi. Note that, by the default settings, Gurobi actually uses a concurrent optimization strategy to solve LPs, which runs multiple classical LP solvers (the primal/dual simplex method and the barrier method) on multiple threads simultaneously and chooses the one that finishes first.

We shall conduct the experiments as follows. In subsection \ref{subsecnum}, we test different methods on synthetic data to show their computational performance in terms of accuracy and speed. In subsection \ref{secMNIST}, we test on the MNIST data set to visualize the quality of results obtained by each method. In subsection \ref{secFree}, we conduct some experiments for the free support case.
A summary of our experiments is given in subsection \ref{secNumsum}.

\subsection{Experiments on Synthetic Data}\label{subsecnum}

In this subsection, we generate a set of discrete probability distributions $\{\mathcal{P}^{(t)}\}_{t=1}^N$ with $\mathcal{P}^{(t)}=\big{\{}(a_i^{(t)},\,\bm{q}_i^{(t)}) \in \mathbb{R}_+\times\mathbb{R}^d: i = 1, \cdots, m_t\big{\}}$ and $\sum^{m_t}_{i=1}a_i^{(t)}=1$, and then apply different methods to solve problem \eqref{subpro1} to compute a Wasserstein barycenter $\mathcal{P}=\big{\{}(w_i,\,\bm{x}_i)\in\mathbb{R}_+\times\mathbb{R}^d: i = 1, \cdots, m\big{\}}$, where $m$ and $(\bm{x}_1,\cdots,\bm{x}_m)$ are pre-specified. Specifically, we set $d=3$, $\gamma_1=\cdots=\gamma_N=\frac{1}{N}$ and $m_1=\cdots=m_N=m'$ for convenience, and choose different $(N, m, m')$. Then, given each triple $(N,m,m')$, we randomly generate a trial in the following three cases.
\begin{itemize}
\item \textbf{Case 1.} \textbf{Each distribution has different \textit{dense} weights (all weights are nonzero) and different support points.} In this case, we first generate the support points $\{\bm{q}_i^{(t)}:i=1,\cdots,m', \,t=1,\cdots,N\}$ whose entries are drawn from a Gaussian mixture distribution via the following {\sc Matlab} commands:
    \vspace{-2mm}
    \begin{verbatim}
    gm_num = 5; gm_mean = [-20; -10; 0; 10; 20];
    sigma = zeros(1,1,gm_num); sigma(1,1,:) = 5*ones(gm_num,1);
    gm_weights = rand(gm_num,1);
    distrib = gmdistribution(gm_mean, sigma, gm_weights);
    \end{verbatim}
    \vspace{-6mm}
    Next, for each $t$, we generate an associated weight vector $(a_1^{(t)}, \cdots, a_{m'}^{(t)})$ whose entries are drawn from the standard uniform distribution on the open interval $(0,1)$, and then normalize it so that $\sum^{m'}_{i=1}a_i^{(t)}=1$. After generating all $\{\mathcal{P}^{(t)}\}_{t=1}^N$, we use the $k$-means\footnote{In our experiments, we call the {\sc Matlab} function ``\texttt{kmeans}", which is built in statistics and machine learning toolbox.} method to choose $m$ points from $\{\bm{q}_i^{(t)}:i=1,\cdots,m', \,t=1,\cdots,N\}$ to be the support points of the barycenter.

\item \textbf{Case 2.} \textbf{Each distribution has different \textit{sparse} weights (most of weights are zeros) and different support points.} In this case, we also generate the support points $\{\bm{q}_i^{(t)}:i=1,\cdots,m', \,t=1,\cdots,N\}$ whose entries are drawn from a Gaussian mixture distribution as in \textbf{Case 1}. Next, for each $t$, we choose a subset $\mathcal{S}_t\subset\{1, \cdots, m'\}$ of size $s$ uniformly at random and generate an $s$-sparse weight vector $(a_1^{(t)}, \cdots, a_{m'}^{(t)})$, which has uniformly distributed entries in the interval $(0,1)$ on $\mathcal{S}_t$ and zeros on $\mathcal{S}^c_t$. Then, we normalize it so that $\sum^{m'}_{i=1}a_i^{(t)}=1$. The number $s$ is set to be $\lfloor m' \times {\tt sr} \rfloor$, where ${\tt sr}$ denotes the sparsity ratio and $\lfloor a \rfloor$ denotes the greatest integer less than or equal to $a$. The number $m$ is set to be larger than $s$. The support points of the barycenter are chosen from $\{\bm{q}_i^{(t)} : a_i^{(t)}\neq0, \,i=1,\cdots,m', \,t=1,\cdots,N\}$ by the $k$-means method. Note that, in this case, one can solve a smaller problem \eqref{redsubpro1} to obtain an optimal solution of \eqref{subpro1}; see Remark \ref{rem}.

\item \textbf{Case 3.} \textbf{Each distribution has different \textit{dense} weights (all weights are nonzero), but has the same support points.} In this case, we set $m=m'$ and generate the points $(\bm{q}_1,\cdots,\bm{q}_{m})$ whose entries are drawn from a Gaussian mixture distribution as in \textbf{Case 1}. Then, all distributions $\{\mathcal{P}^{(t)}\}_{t=1}^N$ and the barycenter use $(\bm{q}_1,\cdots,\bm{q}_{m})$ as the support points. Next, for each $t$, we generate an associated weight vector $(a_1^{(t)}, \cdots, a_{m}^{(t)})$ whose entries are drawn from the standard uniform distribution on the open interval $(0,1)$, and then normalize it so that $\sum^{m}_{i=1}a_i^{(t)}=1$.
\end{itemize}

Tables \ref{ResTableDS},\,\ref{ResTableDS_sparse},\,\ref{ResTableSS} present numerical results of different methods for \textbf{Cases 1},\,\textbf{2},\,\textbf{3}, respectively, where we use different choices of $(N, \,m, \,m')$ and different sparsity ratio ${\tt sr}$. In this part of experiments, we set $\mathrm{Tol}_{\mathrm{sgs}}=\mathrm{Tol}_{\mathrm{b}}=10^{-5}$ and $\mathrm{Tol}_{\mathrm{ibp}}=10^{-8}$ for termination. We also set the maximum numbers of iterations for sGS-ADMM, BADMM and IBP to 3000, 3000, 10000, respectively. In each table, ``normalized obj" denotes the normalized objective value defined by $\frac{\left|\mathcal{F}(\{\Pi^{(t),*}\}) - \mathcal{F}_{\mathrm{gu}}\right|}{\mathcal{F}_{\mathrm{gu}}}$, where $\mathcal{F}(\{\Pi^{(t),*}\}):={\textstyle\sum^N_{t=1}}\langle D^{(t)}, \,\Pi^{(t),*} \rangle$ with $(\bm{w}^*,\,\{\Pi^{(t),*}\})$ being the terminating solution obtained by each algorithm and $\mathcal{F}_{\mathrm{gu}}$ denotes the objective value obtained by Gurobi; ``feasibility" denotes the value of
\begin{equation*}
\eta_{\mathrm{feas}}\big{(}\bm{w}^*,\{\Pi^{(t),*}\}\big{)}:=\max\big{\{}\eta_3\big{(}\bm{w}^*, \{\Pi^{(t),*}\}\big{)}, \,\eta_4\big{(}\{\Pi^{(t),*}\}\big{)},\,\eta_7(\bm{w}^*), \,\eta_8\big{(}\{\Pi^{(t),*}\}\big{)}\big{\}},
\end{equation*}
which is used to measure the deviation of the terminating solution from the feasible set; ``time" denotes the computational time (in seconds); ``iter" denotes the number of iterations. All results presented are the average of 10 independent trials.

One can observe from Tables \ref{ResTableDS},\,\ref{ResTableDS_sparse},\,\ref{ResTableSS} that our sGS-ADMM performs much better than BADMM and IBP ($\varepsilon=0.001$) in the sense that it always returns an objective value considerably closer to that of Gurobi while achieving comparable feasibility accuracy in less computational time. For IBP with $\varepsilon\in\{0.1,0.01\}$, it always converges faster and achieves better feasibility accuracy, but it gives a rather poor objective value, which means that the solution obtained is rather crude. Although a small $\varepsilon=0.001$ can give a better approximation, it may also lead to the numerical instability. The \textit{log-sum-exp} stabilization trick can be used to ameliorate this issue. However, with this trick, IBP (see \eqref{iterschemeIBPlog}) must give up some computational efficiency in matrix-vector multiplications and require many additional exponential evaluations that are typically time-consuming. Moreover, when $\varepsilon$ is small, the convergence of IBP becomes quite slow, as evident in three tables. For BADMM, it can give an objective value close to that of Gurobi. However, it takes much more time and its feasibility accuracy is the worst for most cases. Thus, the performance of BADMM is still not good enough. Moreover, the convergence of BADMM is still unknown. For Gurobi, when $N$, $m$ and $m'$ are relatively small, it can solve the problem highly efficiently. However, when the problem size becomes larger, Gurobi would take much more time. As an example, for the case where $(N,\,m,\,m')=(100,\,300,\,200)$ in Table \ref{ResTableDS}, one would need to solve a large-scale LP containing 6000300 nonnegative variables and 50001 equality constraints. In this case, we see that Gurobi is about 20 times slower than our sGS-ADMM.

\begin{table}[ht]
\setlength{\belowcaptionskip}{7pt}
\captionsetup{width=14.5cm}
\renewcommand\arraystretch{1.15}
\caption{Numerical results on synthetic data for \textbf{Case 1}. In this case, each distribution has different \textit{dense} weights and different support points. In the table, ``sGS" stands for sGS-ADMM; ``BA" stands for BADMM; ``IBP1" stands for IBP with $\varepsilon=0.1$; ``IBP2" stands for IBP with $\varepsilon=0.01$; ``IBP3" stands for IBP with $\varepsilon=0.001$.}\label{ResTableDS}
\centering \tabcolsep 3pt
\scriptsize{
\begin{tabular}{|lll|ccccc|rrrrrr|}
\hline
$N$ & $m$ & $m'$ &\multicolumn{1}{c}{sGS} &\multicolumn{1}{c}{BA}  &  \multicolumn{1}{c}{IBP1}&\multicolumn{1}{c}{IBP2} &\multicolumn{1}{c|}{IBP3}& \multicolumn{1}{c}{Gurobi} &\multicolumn{1}{c}{sGS} &\multicolumn{1}{c}{BA}  & \multicolumn{1}{c}{IBP1}&\multicolumn{1}{c}{IBP2} &\multicolumn{1}{c|}{IBP3} \\
\hline
&&& \multicolumn{5}{c|}{normalized obj} & \multicolumn{6}{c|}{feasibility} \\
\hline
20 &100&100&1.17e-4&5.54e-5&1.17e+0&7.09e-2&3.95e-2 &1.05e-15&1.40e-5&2.00e-4&3.97e-9&3.92e-8&1.22e-4 \\
20 &200&100&2.45e-4&1.18e-4&1.30e+0&9.98e-2&6.60e-2 &9.60e-16&1.39e-5&2.61e-4&2.68e-9&2.08e-8&3.91e-5 \\
20 &200&200&4.01e-4&1.05e-3&2.21e+0&1.28e-1&4.70e-2 &2.41e-7&1.39e-5&3.07e-4&3.66e-9&2.63e-8&4.29e-5 \\
20 &300&200&4.65e-4&1.53e-3&2.33e+0&1.56e-1&6.61e-2 &1.97e-7&1.41e-5&3.67e-4&2.66e-9&1.08e-8&1.45e-5 \\
50 &100&100&9.85e-5&1.20e-4&1.14e+0&6.40e-2&3.46e-2 &2.03e-7&1.40e-5&2.92e-4&7.61e-9&1.30e-7&1.76e-4 \\
50 &200&100&1.57e-4&1.30e-4&1.24e+0&8.93e-2&5.76e-2 &1.25e-7&1.41e-5&3.99e-4&5.83e-9&8.13e-8&1.01e-4 \\
50 &200&200&2.52e-4&1.29e-3&2.09e+0&1.20e-1&4.22e-2 &1.76e-7&1.41e-5&4.60e-4&4.73e-9&3.63e-8&7.31e-5 \\
50 &300&200&4.02e-4&1.93e-3&2.21e+0&1.41e-1&5.74e-2 &4.34e-7&1.40e-5&5.58e-4&3.81e-9&3.89e-8&3.07e-5 \\
100&100&100&2.12e-4&1.35e-4&1.11e+0&6.24e-2&3.39e-2 &2.48e-7&1.45e-5&3.63e-4&7.56e-9&9.03e-8&2.55e-4 \\
100&200&100&3.32e-4&1.99e-4&1.21e+0&8.65e-2&5.68e-2 &1.89e-7&1.43e-5&5.10e-4&6.16e-9&5.23e-8&1.08e-4 \\
100&200&200&5.15e-4&1.35e-3&2.11e+0&1.21e-1&4.35e-2 &3.42e-7&1.51e-5&5.89e-4&6.12e-9&7.69e-8&8.21e-5 \\
100&300&200&6.56e-4&2.04e-3&2.21e+0&1.40e-1&5.53e-2 &5.14e-7&1.47e-5&7.24e-4&5.00e-9&6.11e-8&3.88e-5 \\
\hline
&&& \multicolumn{5}{c|}{iter} & \multicolumn{6}{c|}{time (in seconds)} \\
\hline
20&100&100 &2595&3000&112&2965&10000  &1.84&3.23&33.27&0.14&3.45&24.91  \\
20&200&100 &2495&3000&107&1761&10000  &6.67&8.48&69.20&0.25&3.85&47.51  \\
20&200&200 &2585&3000&103&2049&10000  &10.56&19.24&139.12&0.50&9.39&98.31  \\
20&300&200 &2465&3000&102&1505&10000  &23.68&28.14&208.91&0.76&10.59&152.04  \\
50&100&100 &2930&3000&112&4440&10000  &9.21&13.18&85.33&0.33&12.30&60.53  \\
50&200&100 &2820&3000&110&2712&10000  &53.21&27.36&175.70&0.68&15.94&127.48  \\
50&200&200 &2900&3000&104&2472&10000  &72.73&56.66&341.90&1.30&29.45&250.09  \\
50&300&200 &2840&3000&103&1850&10000  &299.94&85.01&517.10&1.95&33.35&376.42  \\
100&100&100&2985&3000&117&5398&10000  &9.89&28.92&173.16&0.74&32.72&127.03  \\
100&200&100&2980&3000&110&2937&10000  &31.03&58.46&347.86&1.40&35.72&254.26  \\
100&200&200&3000&3000&105&2730&10000  &63.72&117.03&690.55&2.61&64.81&503.19  \\
100&300&200&3000&3000&102&1923&10000  &3703.33&178.99&1032.84&3.80&68.31&756.29  \\
\hline
\end{tabular}}
\end{table}

\begin{table}[ht]
\setlength{\belowcaptionskip}{7pt}
\captionsetup{width=14.5cm}
\renewcommand\arraystretch{1.15}
\caption{Numerical results on synthetic data for \textbf{Case 2}. In this case, each distribution has different \textit{sparse} weights and different support points. In the table, ``sGS" stands for sGS-ADMM; ``BA" stands for BADMM; ``IBP1" stands for IBP with $\varepsilon=0.1$; ``IBP2" stands for IBP with $\varepsilon=0.01$; ``IBP3" stands for IBP with $\varepsilon=0.001$.}\label{ResTableDS_sparse}
\centering \tabcolsep 2.4pt
\scriptsize{
\begin{tabular}{|llll|ccccc|rrrrrr|}
\hline
$N$ & $m$ & $m'$ & $sr$ &
\multicolumn{1}{c}{sGS} &\multicolumn{1}{c}{BA}  &  \multicolumn{1}{c}{IBP1}&\multicolumn{1}{c}{IBP2}&\multicolumn{1}{c|}{IBP3}& \multicolumn{1}{c}{Gurobi} &\multicolumn{1}{c}{sGS} &\multicolumn{1}{c}{BA}  & \multicolumn{1}{c}{IBP1}&\multicolumn{1}{c}{IBP2}&\multicolumn{1}{c|}{IBP3} \\
\hline
&&&& \multicolumn{5}{c|}{normalized obj} & \multicolumn{6}{c|}{feasibility} \\
\hline
50 &50 &500 &0.1&4.22e-5&1.58e-4&5.52e-1&3.54e-2&2.60e-2 &9.22e-16&1.45e-5&2.67e-4&1.67e-8&5.90e-7&5.46e-4 \\
50 &100&500 &0.2&8.58e-5&1.38e-4&1.14e+0&6.19e-2&3.30e-2 &9.28e-8&1.40e-5&2.92e-4&7.05e-9&6.29e-8&1.76e-4 \\
50 &100&1000&0.1&1.02e-4&1.51e-4&1.16e+0&6.47e-2&3.49e-2 &8.89e-8&1.41e-5&2.76e-4&8.08e-9&7.59e-8&1.57e-4 \\
50 &200&1000&0.2&3.21e-4&1.30e-3&2.12e+0&1.22e-1&4.38e-2 &8.11e-8&1.41e-5&4.65e-4&4.26e-9&3.69e-8&6.65e-5 \\
100&50 &500 &0.1&6.26e-5&9.86e-5&5.62e-1&3.53e-2&2.42e-2 &3.36e-8&1.49e-5&2.96e-4&2.00e-8&2.73e-7&6.47e-4 \\
100&100&500 &0.2&1.93e-4&1.68e-4&1.14e+0&6.08e-2&3.22e-2 &2.36e-15&1.48e-5&3.65e-4&7.97e-9&8.39e-7&2.52e-4 \\
100&100&1000&0.1&1.89e-4&1.56e-4&1.13e+0&6.07e-2&3.15e-2 &1.79e-8&1.46e-5&3.62e-4&9.97e-9&8.65e-7&2.39e-4 \\
100&200&1000&0.2&6.04e-4&1.29e-3&2.12e+0&1.22e-1&4.32e-2 &3.19e-7&1.50e-5&5.84e-4&5.41e-9&7.17e-8&7.40e-5 \\
200&50 &500 &0.1&1.31e-4&9.33e-5&5.63e-1&3.56e-2&2.38e-2 &3.43e-8&1.51e-5&3.54e-4&3.54e-8&8.22e-7&7.21e-4 \\
200&100&500 &0.2&4.20e-4&1.61e-4&1.12e+0&6.01e-2&3.23e-2 &1.06e-7&1.56e-5&4.39e-4&7.80e-9&2.50e-7&3.19e-4 \\
200&100&1000&0.1&3.93e-4&1.65e-4&1.12e+0&6.16e-2&3.29e-2 &1.97e-7&1.57e-5&4.35e-4&1.42e-8&3.25e-7&3.27e-4 \\
200&200&1000&0.2&1.27e-3&1.35e-3&2.09e+0&1.20e-1&4.34e-2 &3.09e-7&1.61e-5&7.25e-4&7.78e-9&2.31e-7&1.12e-4 \\
\hline
&&&& \multicolumn{5}{c|}{iter} & \multicolumn{6}{c|}{time (in seconds)} \\
\hline
50&50&500&0.1   &2850&3000&147&7790&10000  &1.66&1.49&12.78&0.05&2.43&10.46  \\
50&100&500&0.2  &2965&3000&110&3098&10000  &9.19&13.06&83.96&0.33&8.60&60.05  \\
50&100&1000&0.1 &2945&3000&109&4071&10000  &9.13&12.98&84.11&0.32&11.29&59.90  \\
50&200&1000&0.2 &2885&3000&104&2294&10000  &75.44&55.95&337.89&1.29&27.23&249.11  \\
100&50&500&0.1  &2965&3000&137&6915&10000  &1.86&5.64&41.85&0.21&10.25&31.29  \\
100&100&500&0.2 &3000&3000&111&4520&10000  &10.40&28.71&171.21&0.70&27.18&126.31  \\
100&100&1000&0.1&3000&3000&118&5675&10000  &11.01&28.76&171.31&0.74&34.22&126.46  \\
100&200&1000&0.2&3000&3000&104&2985&10000  &63.89&117.93&674.10&2.57&70.33&499.95  \\
200&50&500&0.1  &3000&3000&154&8143&10000  &3.98&13.56&85.42&0.48&24.49&63.71  \\
200&100&500&0.2 &3000&3000&126&5600&10000  &27.66&57.73&339.87&1.60&68.39&254.00  \\
200&100&1000&0.1&3000&3000&116&5764&10000  &31.36&57.75&340.34&1.47&70.34&254.03  \\
200&200&1000&0.2&3000&3000&104&3107&10000  &143.95&224.84&1366.46&5.14&146.97&1010.56  \\
\hline
\end{tabular}}
\end{table}

\begin{table}[ht]
\setlength{\belowcaptionskip}{7pt}
\captionsetup{width=14.5cm}
\renewcommand\arraystretch{1.15}
\caption{Numerical results on synthetic data for \textbf{Case 3}. In this case, each distribution has different \textit{dense} weights, but has the same support points. In the table, ``sGS" stands for sGS-ADMM; ``BA" stands for BADMM; ``IBP1" stands for IBP with $\varepsilon=0.1$; ``IBP2" stands for IBP with $\varepsilon=0.01$; ``IBP3" stands for IBP with $\varepsilon=0.001$.}\label{ResTableSS}
\centering \tabcolsep 3pt
\scriptsize{
\begin{tabular}{|lll|ccccc|rrrrrr|}
\hline
$N$ & $m$ & $m'$ &
\multicolumn{1}{c}{sGS} &\multicolumn{1}{c}{BA}  &  \multicolumn{1}{c}{IBP1}&\multicolumn{1}{c}{IBP2}&\multicolumn{1}{c|}{IBP3}& \multicolumn{1}{c}{Gurobi} &\multicolumn{1}{c}{sGS} &\multicolumn{1}{c}{BA}  & \multicolumn{1}{c}{IBP1}&\multicolumn{1}{c}{IBP2}&\multicolumn{1}{c|}{IBP3} \\
\hline
&&& \multicolumn{5}{c|}{normalized obj} & \multicolumn{6}{c|}{feasibility} \\
\hline
20 &50 &50 &1.68e-4&4.08e-4&1.02e+0&2.23e-2&4.48e-3 &6.79e-16&1.42e-5&2.22e-4&1.28e-8&3.71e-6&1.17e-3 \\
20 &100&100&1.84e-4&4.12e-4&2.12e+0&6.26e-2&1.91e-3 &1.97e-8&1.43e-5&2.92e-4&1.24e-8&1.49e-6&5.11e-4 \\
20 &200&200&8.12e-4&2.80e-3&4.34e+0&1.72e-1&1.59e-3 &2.55e-7&1.40e-5&4.38e-4&4.10e-9&1.25e-6&2.49e-4 \\
50 &50 &50 &9.73e-5&6.26e-4&1.02e+0&2.18e-2&3.84e-3 &4.11e-8&1.63e-5&3.18e-4&1.98e-8&1.18e-5&1.52e-3 \\
50 &100&100&2.47e-4&3.91e-4&2.07e+0&6.04e-2&2.48e-3 &9.32e-8&1.64e-5&4.32e-4&1.25e-8&2.16e-6&8.39e-4 \\
50 &200&200&6.17e-4&2.81e-3&4.23e+0&1.65e-1&1.50e-3 &3.37e-7&1.49e-5&6.46e-4&6.70e-9&1.37e-7&3.77e-4 \\
100&50 &50 &1.39e-4&2.72e-4&1.02e+0&2.15e-2&3.95e-3 &1.17e-7&1.85e-5&4.05e-4&3.16e-8&1.14e-5&1.95e-3 \\
100&100&100&3.85e-4&4.13e-4&2.07e+0&6.00e-2&2.49e-3 &1.88e-7&1.73e-5&5.27e-4&1.08e-8&4.97e-6&1.05e-3 \\
100&200&200&1.06e-3&2.94e-3&4.19e+0&1.63e-1&1.46e-3 &3.79e-7&1.65e-5&8.17e-4&6.30e-9&7.26e-7&4.91e-4 \\
200&50 &50 &2.45e-4&2.87e-4&1.02e+0&2.15e-2&3.65e-3 &5.21e-8&1.90e-5&4.43e-4&1.87e-7&1.31e-5&2.14e-3 \\
200&100&100&7.75e-4&4.08e-4&2.05e+0&5.91e-2&2.59e-3 &6.45e-8&1.81e-5&6.45e-4&1.87e-8&5.66e-6&1.28e-3 \\
200&200&200&2.33e-3&2.96e-3&4.15e+0&1.61e-1&1.34e-3 &3.43e-7&1.73e-5&1.02e-3&8.17e-9&9.32e-7&5.90e-4 \\
\hline
&&& \multicolumn{5}{c|}{iter} & \multicolumn{6}{c|}{time (in seconds)} \\
\hline
20&50&50   &2895&3000&316&8465&10000  &0.29&0.72&5.59&0.03&0.79&5.22  \\
20&100&100 &2925&3000&225&6383&10000  &1.69&3.68&33.69&0.03&0.75&24.68  \\
20&200&200 &2765&3000&157&6037&10000  &9.90&21.28&139.94&0.04&1.11&98.05  \\
50&50&50   &3000&3000&286&9815&10000  &1.34&1.91&13.48&0.04&1.23&10.98  \\
50&100&100 &3000&3000&226&8759&10000  &9.41&14.11&85.93&0.05&1.81&60.24  \\
50&200&200 &2995&3000&161&5603&10000  &74.60&62.07&343.52&0.08&2.17&250.74  \\
100&50&50  &3000&3000&428&9685&10000  &1.98&6.28&42.33&0.09&1.89&31.40  \\
100&100&100&3000&3000&330&9182&10000  &11.35&30.30&173.30&0.13&3.25&126.06  \\
100&200&200&3000&3000&157&7767&10000  &51.47&125.53&685.99&0.14&4.79&501.63  \\
200&50&50  &3000&3000&399&9876&10000  &4.20&13.78&86.35&0.15&3.44&63.13  \\
200&100&100&3000&3000&238&9662&10000  &29.93&58.16&343.43&0.15&5.16&252.98  \\
200&200&200&3000&3000&157&9107&10000  &135.97&225.66&1370.98&0.23&9.10&1003.73  \\
\hline
\end{tabular}}
\end{table}

We next follow \citep[Section 3.4]{cp2016a} to conduct a simple example to visually show the qualities of the barycenter $\bm{w}^*$ and transport plans $\{\Pi^{(t),*}\}$ computed by different algorithms. Consider two one-dimensional continuous Gaussian distributions $\mathcal{N}(\mu_1, \sigma_1^2)$ and $\mathcal{N}(\mu_2, \sigma_2^2)$. It is known from \citep[Section 6.2]{ac2011barycenters} and \citep[Example 1.7]{m1997a} that their 2-Wasserstein barycenter is the Gaussian distribution $\mathcal{N}\left(\frac{\mu_1+\mu_2}{2}, \big{(}\frac{\sigma_1+\sigma_2}{2}\right)^2\big{)}$. Based on this fact, we discretize two Gaussian distributions $\mathcal{N}(-2, \big{(}\frac{1}{4}\big{)}^2)$ and $\mathcal{N}(2, 1)$, and then apply different algorithms to compute their barycenter, which is expected to be close to the discretization of the true barycenter $\mathcal{N}(0, \big{(}\frac{5}{8}\big{)}^2)$. The discretization is performed on the interval $[-4, 5]$ with $n$ uniform grids. Since this part of experiments is not intended for comparing speed, we shall use tighter tolerances, say, $\mathrm{Tol}_{\mathrm{sgs}}=\mathrm{Tol}_{\mathrm{b}}=10^{-6}$ and $\mathrm{Tol}_{\mathrm{ibp}}=10^{-10}$, and set the maximum numbers of iterations for all algorithms to 20000. Figure \ref{1dbarycenters}(a) shows the barycenters computed by different algorithms for $n=500$. From this figure, we see that the barycenter computed by Gurobi oscillates wildly. A similar result has also been observed in \citep[Section 3.4]{cp2016a}. The possible reason for this phenomenon is that the LP \eqref{subpro1} has multiple solutions and Gurobi using a simplex method may not find a ``smooth" one. IBP always finds a ``smooth'' solution thanks to the entropic regularization in the objective. A smaller $\varepsilon$ (say, 0.001) indeed gives a better approximation. On the other hand, our sGS-ADMM and BADMM are also able to find a ``smooth" barycenter, although they are designed to solve the original LP. This could be due to the fact that these two algorithms are developed based on the augmented Lagrangian function or its variants, and they implicitly have a `smoothing' regularization (due to the penalty or proximal term) in each subproblem. In particular, just as IBP with $\varepsilon=0.001$, the barycenter computed by the sGS-ADMM can match the true barycenter almost exactly. We also show the transport plans for $n=500$ in Figure \ref{1dbarycenters}(b). One can see that the transport plans computed by sGS-ADMM are more similar to those computed by Gurobi, while the transport plans computed by IBP are more blurry. Consequently, these two figures clearly demonstrate the superior quality of the solution obtained by our sGS-ADMM.

\begin{figure}[ht]
\captionsetup{width=14cm}
\centering
\subfigure[Distributions and barycenters]{\includegraphics[width=12cm]{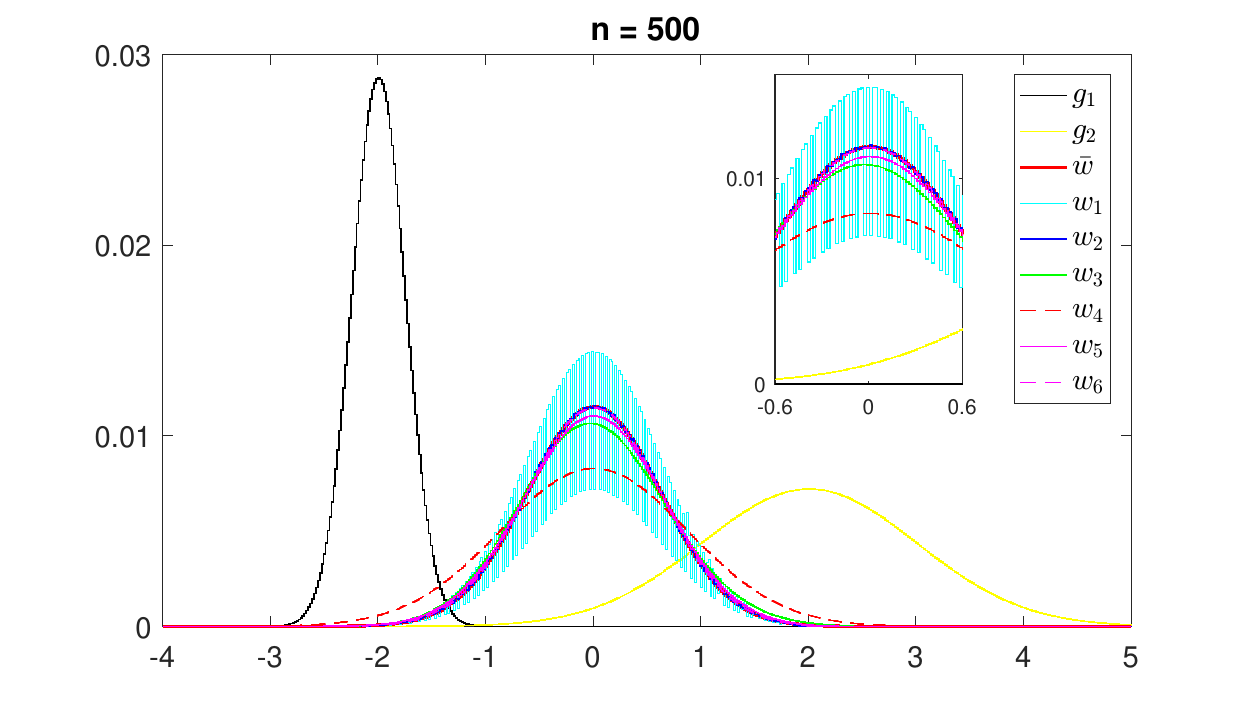}}  \\
\subfigure[Transport plans for $n=500$]{\includegraphics[width=13cm]{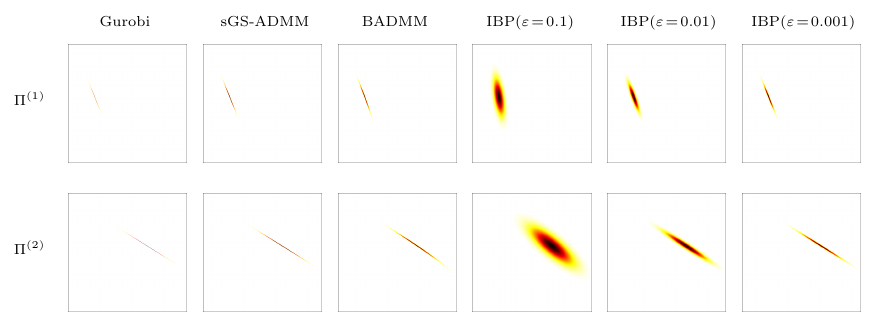}}
\caption{In figure (a): $g_1$ stands for the discretization of $\mathcal{N}(-2, \big{(}\frac{1}{4}\big{)}^2)$; $g_2$ stands for the discretization of $\mathcal{N}(2, 1)$; $\bar{w}$ stands for the discretization of the true barycenter $\mathcal{N}(0, \big{(}\frac{5}{8}\big{)}^2)$ of $g_1$ and $g_2$; $w_1$, $w_2$, $w_3$ stand for the barycenter computed by Gurobi, sGS-ADMM and  BADMM, respectively; $w_4$, $w_5$, $w_6$ stand for the barycenter computed by IBP with $\varepsilon=0.1, 0.01, 0.001$, respectively. The discretization is performed on the interval $[-4, 5]$ with $n$ uniform grids. In figure (b): $\Pi^{(1)}$ (resp. $\Pi^{(2)}$) stands for the transport plan between the barycenter and $g_1$ (resp. $g_2$).}\label{1dbarycenters}
\end{figure}

To further compare the performances of Gurobi and our sGS-ADMM, we conduct more experiments on synthetic data for \textbf{Case 1}, where we fix two of three dimensions $m$, $m'$, $N$ and vary the third one. In this part of experiments, we use $\mathrm{Tol}_{\mathrm{sgs}}=10^{-5}$ to terminate our sGS-ADMM without setting the maximum iteration number. Figure \ref{FigDiff} shows the computational results of the two algorithms over a range of $m$, $m'$ or $N$, and each value is an average over 10 independent trials. From the results, one can see that our sGS-ADMM always returns a similar objective value as Gurobi and has a reasonably good feasibility accuracy. For the computational time, our sGS-ADMM increases approximately linearly with respect to $m$, $m'$ or $N$ individually, while Gurobi increases much more rapidly. This is because the solution methods used in Gurobi (the primal/dual simplex method and the barrier method) are no longer efficient enough and may consume too much memory (due to the Cholesky factorization of a huge coefficient matrix) when the problem size becomes large, although Gurobi already uses a parallel implementation to exploit multiple processors. Moreover, Gurobi may lack robustness, especially for solving large-scale problems. Indeed, as observed from our experiments, the computational times taken by Gurobi can vary a lot among the 10 randomly generated instances, especially $m$, $m'$ or $N$ becomes large. On the other hand, as discussed in Section \ref{sectalg}, the main computational complexity of our sGS-ADMM at each iteration is $\mathcal{O}(Nmm')$. Hence, when two of $m$, $m'$, $N$ are fixed, the total computational cost of our sGS-ADMM is approximately linear with respect to the remaining one, as shown in Figure \ref{FigDiff}. This then highlights another advantage of our method. In addition, although our sGS-ADMM takes advantage of many efficient built-in functions (e.g., matrix multiplication and addition) in {\sc Matlab} that can execute on multiple computational threads, we believe that there is still ample room for improving our sGS-ADMM with a dedicated parallel implementation on a suitable computing platform other than {\sc Matlab}. But we will leave this topic as future research.

\begin{figure}[p]
\centering
\subfigure[$m$ varies with $N=50$ and $m'=100$]{
\begin{minipage}[t]{0.5\textwidth}
\centering
\includegraphics[height=6cm]{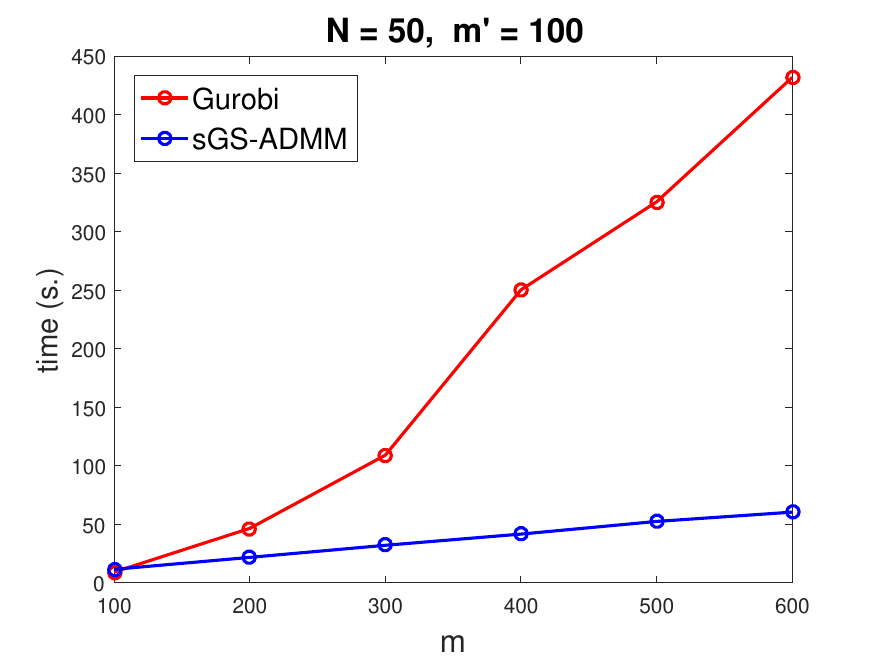}
\end{minipage}
\begin{minipage}[t]{0.45\textwidth}
\vspace{-4.5cm}\scriptsize{
\begin{tabular}{|l|cc|cc|}
\hline
& \multicolumn{2}{c|}{normalized obj} & \multicolumn{2}{c|}{feasibility} \\
\hline
\multicolumn{1}{|c|}{$m$} & {\tiny Gurobi} & {\tiny sGS-ADMM} & {\tiny Gurobi} & {\tiny sGS-ADMM}  \\
\hline
100 & 0 & 9.19e-05 & 2.03e-07 & 1.40e-05  \\
200 & 0 & 1.57e-04 & 1.25e-07 & 1.41e-05  \\
300 & 0 & 3.03e-04 & 1.56e-07 & 1.40e-05  \\
400 & 0 & 3.91e-04 & 1.76e-07 & 1.41e-05  \\
500 & 0 & 5.20e-04 & 1.76e-07 & 1.40e-05  \\
600 & 0 & 5.74e-04 & 1.71e-07 & 1.41e-05  \\
\hline
\end{tabular}}
\end{minipage}}

\subfigure[$m'$ varies with $N=50$ and $m=500$]{
\begin{minipage}[t]{0.5\textwidth}
\centering
\includegraphics[height=6cm]{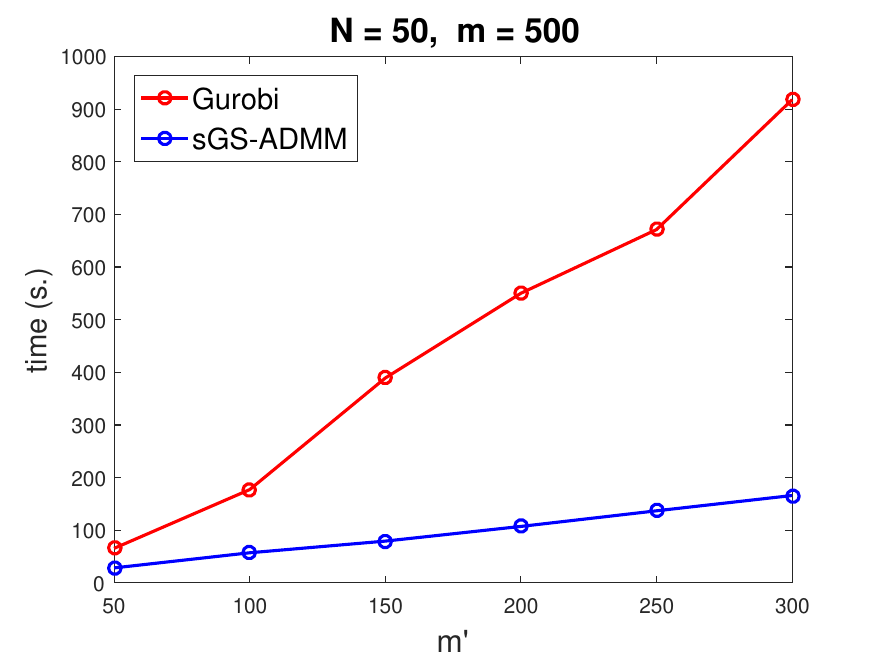}
\end{minipage}
\begin{minipage}[t]{0.45\textwidth}
\vspace{-4.5cm}\scriptsize{
\begin{tabular}{|l|cc|cc|}
\hline
& \multicolumn{2}{c|}{normalized obj} & \multicolumn{2}{c|}{feasibility} \\
\hline
\multicolumn{1}{|c|}{$m'$} & {\tiny Gurobi} & {\tiny sGS-ADMM} & {\tiny Gurobi} & {\tiny sGS-ADMM}  \\
\hline
50  & 0 & 4.14e-04 & 2.54e-10 & 1.40e-05  \\
100 & 0 & 4.85e-04 & 6.60e-10 & 1.40e-05  \\
150 & 0 & 5.81e-04 & 2.56e-10 & 1.41e-05  \\
200 & 0 & 6.38e-04 & 4.14e-10 & 1.42e-05  \\
250 & 0 & 7.64e-04 & 1.41e-11 & 1.41e-05  \\
300 & 0 & 9.64e-04 & 2.35e-11 & 1.41e-05  \\
\hline
\end{tabular}}
\end{minipage}}

\subfigure[$N$ varies with $m=20$ and $m'=10$]{
\begin{minipage}[t]{0.5\textwidth}
\centering
\includegraphics[height=6cm]{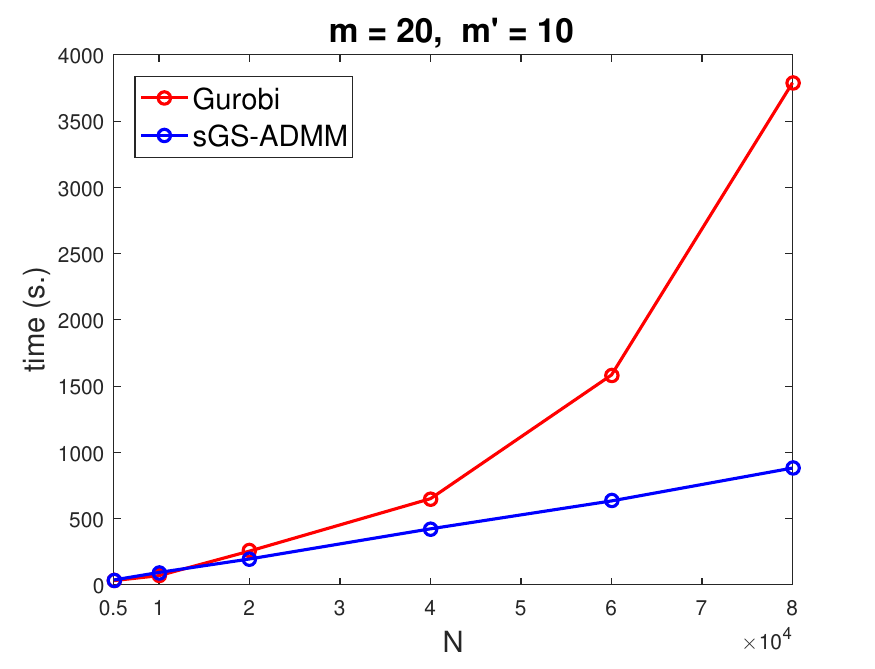}
\end{minipage}
\begin{minipage}[t]{0.45\textwidth}
\vspace{-4.5cm}
\scriptsize{
\begin{tabular}{|l|cc|cc|}
\hline
& \multicolumn{2}{c|}{normalized obj} & \multicolumn{2}{c|}{feasibility} \\
\hline
\multicolumn{1}{|c|}{$N$} & {\tiny Gurobi} & {\tiny sGS-ADMM} & {\tiny Gurobi} & {\tiny sGS-ADMM}  \\
\hline
5000   & 0 & 1.28e-05 & 9.78e-09 & 5.19e-06  \\
10000  & 0 & 1.12e-05 & 2.15e-08 & 4.60e-06  \\
20000  & 0 & 1.03e-05 & 1.39e-08 & 4.19e-06  \\
40000  & 0 & 9.83e-06 & 1.79e-08 & 3.92e-06  \\
60000  & 0 & 9.59e-06 & 1.87e-08 & 3.77e-06  \\
80000  & 0 & 9.35e-06 & 1.92e-08 & 3.78e-06  \\
\hline
\end{tabular}}
\end{minipage}
}
\caption{Comparisons between Gurobi and sGS-ADMM}\label{FigDiff}
\end{figure}

\subsection{Experiments on MNIST}\label{secMNIST}

To better visualize the quality of results obtained by each method, we conduct similar experiments to
\citep[Section 6.1]{cd2014fast} on the MNIST\footnote{Available in \url{http://yann.lecun.com/exdb/mnist/}.} data set \citep{lbbh1998gradient}. Specifically, we randomly select 50 images for each digit ($0\sim9$) and resize each image to $\zeta$ times of its original size of $28 \times 28$, where $\zeta$ is drawn uniformly at random between 0.5 and 2. Then, we randomly put each resized image in a larger $56\times56$ blank image and normalize the resulting image so that all pixel values add up to 1. Thus, each image can be viewed as a discrete distribution supported on grids. We then apply sGS-ADMM, BADMM and IBP with $\varepsilon\in\{0.01,0.001\}$ to compute a Wasserstein barycenter of the resulting images for each digit. The size of barycenter is set to $56\times56$. Note that, since each input image can be viewed as a \textit{sparse} discrete distribution because most of the pixel values are zeros,
one can actually solve a smaller problem \eqref{redsubpro1} to obtain a barycenter; see Remark \ref{rem}. Moreover, for such grid-supported data, an efficient convolutional technique \citep{sgpc2015convolutional} and its stabilized  version \citep[Section 4.1.2]{shbnccps2018wasserstein} have also been used to substantially accelerate IBP and the stabilized IBP, respectively, in our experiments.

The computational results are shown in Figure \ref{figWBSS_con}. One can see that, our sGS-ADMM can provide a clear ``smooth'' barycenter just like IBP with $\varepsilon=0.001$, although it is designed to solve the original LP. This again shows the superior quality of the solution obtained by our sGS-ADMM. Moreover, the results obtained by running sGS-ADMM for 100s are already much better than those obtained by running BADMM for 800s. IBP performs very well on this grid-supported data with smaller $\varepsilon$ leading to sharper barycenters. Here, we would also like to point out that, without the novel convolutional technique, IBP (especially with a small $\varepsilon$) would take much longer time to produce sharper images. Moreover, when using the convolutional technique in IBP, one can no longer take advantage of the sparsity of the distributions and needs to solve the problem on the full grids. This may limit the adoption of the convolutional technique for the case when the distributions are highly sparse (most of weights are zeros) but supported on very dense \textit{or} high dimensional grids. In that case, our sGS-ADMM may be more favorable.

\begin{figure}[p]
\captionsetup{width=14cm}
\centering
\includegraphics[height=19cm]{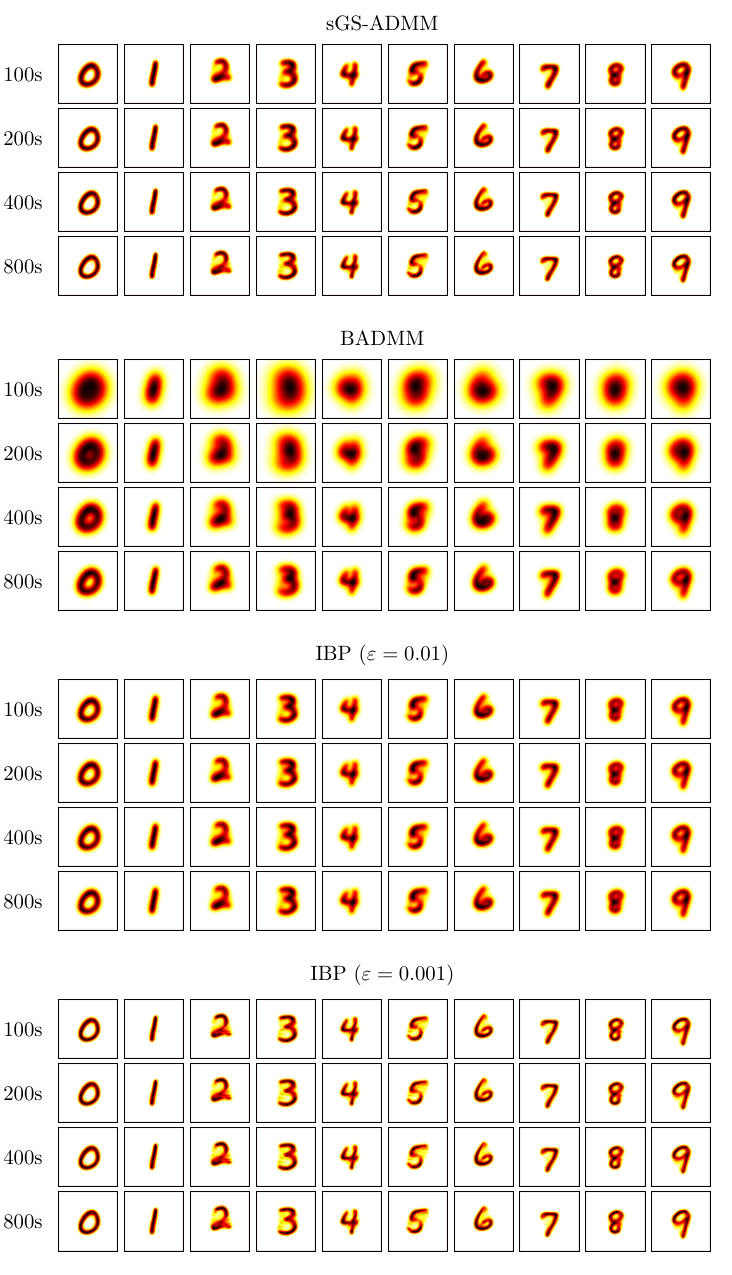}
\caption{Barycenters obtained by running different methods for 100s, 200s, 400s, 800s, respectively. }\label{figWBSS_con}
\end{figure}

\subsection{Experiments for the Free Support Case}\label{secFree}

In this subsection, we briefly compare the performance of different methods used as subroutines in the alternating minimization method for computing a barycenter whose support points are not pre-fixed, i.e., solving problem \eqref{mainpro}; see Remark \ref{RemFree}. The experiments are conducted on the same image data sets\footnote{Available in \url{https://github.com/bobye/d2_kmeans/tree/master/data}} with three different categories (mountains, sky and water) as in \citep{yl2014scaling}. For each category, the data set consists of 1000 discrete distributions and each distribution is obtained by clustering pixel colors of an image in this category \citep[Section 2.3]{lw2008real}. The average number of support points is around 8 and the dimension of each support point is 3. We then compute the barycenter of each data set. The number of support points of the barycenter is chosen as $m=10,\,50$ and the initial $m$ support points are computed as the centroids of clusters obtained by applying $k$-means to all the support points of the given distributions.

The performance of the alternating minimization method for solving problem \eqref{mainpro} naturally depends on the accuracy of the approximate solution obtained for each subproblem (namely, problem \eqref{subpro1}). Basically, a more accurate approximate solution is more likely to guarantee the descent of the objective in problem \eqref{mainpro}, but is also more costly to obtain. Thus, it is nontrivial to design an optimal stopping criterion for the subroutine when solving each subproblem. In our experiments, we simply use the maximum iteration number to terminate the subroutine. For sGS-ADMM and BADMM, we follow \citep{ywwl2017fast} to set the maximum iteration number to 10. For IBP, we set the maximum iteration number to 100 for $\varepsilon=0.01$ and to 1000 for $\varepsilon=0.001$. As observed from our experiments, such maximum numbers are `optimal' for IBP in the sense that an approximate solution having a reasonably good feasibility accuracy can be obtained in less CPU time for most cases. At each iteration, each subroutine is also warm started by the approximate solution obtained in the previous iteration. Finally, we terminate the alternating minimization and return the approximate solution when the relative successive change of the objective in problem \eqref{mainpro} is smaller than $10^{-5}$.

The computational results are reported in Table \ref{ResFreeCase}. One can see that Gurobi performs best in terms of the solution accuracy because it always achieves the lowest objective value and the best feasibility accuracy. However, it is significantly more time-consuming, compared with our sGS-ADMM. In comparison to BADMM and IBP, our sGS-ADMM always returns lower objective values (which are also much closer to those of Gurobi) and better feasibility accuracy within competitive computational time. Similar to the numerical observations in
\citep[Section IV]{ywwl2017fast}, IBP always gives the worse objective values in our experiments. One possible reason is that, due to the entropy regularization, using IBP as the subroutine to approximately solve the subproblem is less likely to ensure the monotonic decrease of the objective values in problem \eqref{mainpro} during the iterations. In our experiments, we have observed that careful tuning of the regularization parameter $\varepsilon$ and intricate adjustment of the corresponding stopping criterion are needed for IBP
to perform well as a subroutine within the alternating minimization method. In view of the above, our sGS-ADMM can be more favorable to be incorporated in the alternating minimization method for handling the free support case. But we should also mention that the computation of barycenters with free supports is still a challenging problem nowadays, which is nonconvex and often presented in large-scale. More in-depth study on applying our sGS-ADMM to this problem is needed and will be left as a future research project.

\begin{table}[ht]
\setlength{\belowcaptionskip}{7pt}
\captionsetup{width=14cm}
\renewcommand\arraystretch{1.1}
\caption{Numerical results for the free support case}\label{ResFreeCase}
\renewcommand\arraystretch{1.15}
\centering \tabcolsep 6pt
\small{
\begin{tabular}{|c|c|ccl|ccl|}
\hline
\multirow{2}{*}{\small{data set}} & \multirow{2}{*}{\small{method}} & \multicolumn{3}{c|}{$m=10$} & \multicolumn{3}{c|}{$m=50$} \\
\cline{3-8}
&& \small{obj} & \small{feasibility} & \small{time(s)} & \small{obj} & \small{feasibility} & \small{time(s)} \\
\hline
\multirow{5}{*}{mountains}
&\small{Gurobi}   &1490.24&1.56e-12&13.46 & 1480.38&5.73e-11&103.39 \\
&\small{sGS-ADMM} &1491.37&2.23e-04&2.25  & 1481.70&1.03e-04&6.09 \\
&\small{BADMM}    &1497.79&3.46e-03&1.21  & 1483.24&5.32e-03&8.73 \\
&\small{IBP($\varepsilon\!=\!0.01$)}
&1509.30&3.77e-04&3.71  & 1503.79&5.84e-04&10.44 \\
&\small{IBP($\varepsilon\!=\!0.001$)}
&1529.99&7.02e-04&37.20 & 1530.69&1.15e-03&311.93 \\
\hline

\multirow{5}{*}{sky}
&\small{Gurobi}   & 1623.42&1.50e-12&14.61  &1612.30&1.26e-12&89.48  \\
&\small{sGS-ADMM} & 1624.60&2.98e-04&2.31   &1614.27&1.12e-04&5.68  \\
&\small{BADMM}    & 1632.14&3.12e-03&1.44   &1633.07&2.66e-02&1.92  \\
&\small{IBP($\varepsilon\!=\!0.01$)}
&1643.13&7.62e-04&2.99  & 1637.36&9.06e-04&7.17\\
&\small{IBP($\varepsilon\!=\!0.001$)}
&1735.32&2.10e-04&26.96 & 1639.74&1.19e-03&244.22 \\
\hline

\multirow{5}{*}{water}
&\small{Gurobi}   & 1620.78&3.87e-13&15.52 & 1611.15&5.05e-11&73.86  \\
&\small{sGS-ADMM} & 1622.20&2.02e-04&2.54  & 1613.16&1.31e-04&4.76  \\
&\small{BADMM}    & 1653.79&8.66e-03&0.27  & 1615.14&1.02e-02&5.17  \\
&\small{IBP($\varepsilon\!=\!0.01$)}
&1644.79&1.34e-04&3.87  & 1635.35&3.77e-04&11.98\\
&\small{IBP($\varepsilon\!=\!0.001$)}
&1734.99&1.40e-04&39.53 & 1646.01&2.05e-03&197.29\\
\hline


\end{tabular}}
\end{table}

\subsection{Summary of Experiments}\label{secNumsum}

From the numerical results reported in the last few subsections, one can see that our sGS-ADMM outperforms the powerful commercial solver Gurobi in terms of the computational time for solving large-scale LPs arising from Wasserstein barycenter problems. Our sGS-ADMM is also much more efficient than another non-regularization-based algorithm BADMM that is designed to solve the primal LP \eqref{repro}. Comparing to IBP on non-grid-supported data, our sGS-ADMM always returns high quality solutions comparable to those obtained by Gurobi but within much shorter computational time. Moreover, our sGS-ADMM is also able to find a ``smooth'' barycenter as in IBP even though we do not modify the LP objective function by adding an entropic regularization. Finally, we would like to emphasize that in contrast to IBP that uses a small $\varepsilon$, our sGS-ADMM does not suffer from numerical instability issues or exceedingly slow convergence speed. Thus, one can easily apply our sGS-ADMM for computing a high quality Wasserstein barycenter without the need to implement sophisticated  stabilization techniques as in the case of IBP.

\section{Concluding Remarks}\label{secconc}

In this paper, we consider the problem of computing a Wasserstein barycenter with pre-specified support points for a set of discrete probability distributions with finite support points. This problem can be modeled as a large-scale linear programming (LP) problem. To solve this LP, we derive its dual problem and then adapt a symmetric Gauss-Seidel based alternating direction method of multipliers (sGS-ADMM) to solve the resulting dual problem. We also establish its global linear convergence without any condition. Moreover, we have designed the algorithm so that all the subproblems involved can be solved exactly and efficiently in a distributed fashion. This makes our sGS-ADMM highly suitable for computing a Wasserstein barycenter on a large data set. Finally, we have conducted detailed numerical experiments on synthetic data sets and image data sets to illustrate the efficiency of our method.

\acks{The authors are grateful to the editor and the anonymous referees for their valuable suggestions and comments, which have helped to improve the quality of this paper. The research of Defeng Sun was supported in part by a start-up research grant from the Hong Kong Polytechnic University. The research of Kim-Chuan Toh was supported in part by the Ministry of Education, Singapore, Academic Research Fund (Grant No. R-146-000-256-114).}

\appendix

\section{An iterative Bregman projection method}\label{apdIBP}

The iterative Bregman projection (IBP) method was adapted by \cite{bccnp2015iterative} to solve the following problem, which introduces an entropic regularization in the original LP \eqref{subpro1}:
\begin{equation}\label{entropypro}
\begin{aligned}
&\min\limits_{\bm{w},\,\{\Pi^{(t)}\}}~{\textstyle\frac{1}{N}\sum^N_{t=1}}\big{(}\langle D^{(t)}, \,\Pi^{(t)} \rangle - \varepsilon E_t(\Pi^{(t)})\big{)}  \\
&\hspace{0.5cm}\mathrm{s.t.} \hspace{0.5cm} \Pi^{(t)}\bm{e}_{m_t} = \bm{w}, ~(\Pi^{(t)})^{\top}\bm{e}_{m} = \bm{a}^{(t)},~\Pi^{(t)} \geq 0, ~~\forall\, t = 1, \cdots, N, \\
&\hspace{1.5cm} \bm{e}^{\top}_m\bm{w} = 1, ~\bm{w} \geq 0,
\end{aligned}
\end{equation}
where the entropic regularization $E_t(\Pi^{(t)})$ is defined as $E_t(\Pi^{(t)})=-\sum^m_{i=1}\sum^{m_t}_{j=1}\pi^{(t)}_{ij}(\log(\pi^{(t)}_{ij})$ $-1)$ for $t=1,\cdots,N$ and $\varepsilon>0$ is a regularization parameter. Let $\Xi_t=\exp(-D^{(t)}/\varepsilon) \in \mathbb{R}^{m\times m_t}$ for $t=1,\cdots,N$. Then, it follows from \citep[Remark 3]{bccnp2015iterative} that IBP for solving \eqref{entropypro} is given by
\begin{equation}\label{iterschemeIBP}
\begin{aligned}
\bm{u}^{(t),k+1} &= \bm{w}^k./\big{(}\Xi_t\bm{v}^{(t),k}\big{)},  \quad t = 1, \cdots, N,  \\
\bm{v}^{(t),k+1} &= \bm{a}^{(t)}./\big{(}\Xi_t^{\top}\bm{u}^{(t),k+1}\big{)}, \quad t = 1, \cdots, N,  \\
\Pi^{(t),k+1} &= \mathrm{Diag}(\bm{u}^{(t),k+1})\,\Xi_t\,\mathrm{Diag}(\bm{v}^{(t),k+1}), \quad t = 1, \cdots, N,
\\
\bm{w}^{k+1} &= \left({\textstyle\prod^N_{t=1}}\big(\bm{u}^{(t),k+1}\odot(\Xi_t\bm{v}^{(t),k+1})\big)\right)^{\frac{1}{N}},
\end{aligned}
\end{equation}
with $\bm{w}^0=\frac{1}{m}\bm{e}_m$ and $\bm{v}^{(t),0}=\bm{e}_{m_t}$ for $t=1,\cdots,N$, where $\mathrm{Diag}(\bm{x})$ denotes the diagonal matrix with the vector $\bm{x}$ on the main diagonal, ``$./$" denotes the entrywise division and ``$\odot$" denotes the entrywise product. Note that the main computational cost in each iteration of the above iterative scheme is $\mathcal{O}(m\sum^N_{t=1}m_t)$. Moreover, when all distributions have the same $m'$ support points, IBP can be implemented highly efficiently with a $\mathcal{O}((m+m')N)$ memory complexity, while sGS-ADMM and BADMM still require $\mathcal{O}(mm'N)$ memory. Specifically, in this case, IBP can avoid forming and storing the large matrix $[\Xi_1,\cdots,\Xi_N]$ (since each $\Xi_t$ is the same) to compute $\Xi_t\bm{v}^{(t),k}$ and $\Xi_t^{\top}\bm{u}^{(t),k+1}$. Thus, IBP can reduce much computational cost and take less time at each iteration. This advantage can be seen in Table \ref{ResTableSS} for $\varepsilon\in\{0.1,0.01\}$. However, we should be mindful that IBP only solves problem \eqref{entropypro} to obtain an approximate solution of the original problem \eqref{subpro1}. Although a smaller $\varepsilon$ can give a better approximation, IBP may become numerically unstable when $\varepsilon$ is too small; see
\citep[Section 1.3]{bccnp2015iterative} for more details. To alleviate this numerical instability, one may carry out the computations in \eqref{iterschemeIBP} in the log domain and use the \textit{log-sum-exp} stabilization trick to avoid underflow/overflow for small values of $\varepsilon$; see \citep[Section 4.4]{pc2019computational} for more details. Specifically, by taking logarithm on both sides of the equations in \eqref{iterschemeIBP} and letting $\tilde{\bm{u}}^{(t),k}:=\varepsilon\log(\bm{u}^{(t),k})$, $\tilde{\bm{v}}^{(t),k}:=\varepsilon\log(\bm{v}^{(t),k})$, $\tilde{\bm{w}}^{(t),k}:=\varepsilon\log(\bm{w}^{(t),k})$ and
$\tilde{\bm{a}}^{(t)}:=\varepsilon\log(\bm{a}^{(t)})$, we obtain after some manipulations that
\begin{equation}\label{iterschemeIBPlog}
\begin{aligned}
\tilde{\bm{u}}^{(t),k+1}
&= \tilde{\bm{w}}^k
+ \tilde{\bm{u}}^{(t),k} - {\textstyle\varepsilon\log\left(\!\left[\sum\limits^{m_t}_{j=1}
\exp\!\left(\frac{\tilde{u}^{(t),k}_i+\tilde{v}^{(t),k}_j-D^{(t)}_{ij}}{\varepsilon}\right)\!\right]_i\right)}, ~~ t = 1, \cdots, N, \\
\tilde{\bm{v}}^{(t),k+1} &= \tilde{\bm{a}}^{(t)}
+ \tilde{\bm{v}}^{(t),k} - {\textstyle\varepsilon\log\left(\!\left[\sum\limits^{m}_{i=1}
\exp\!\left(\frac{\tilde{u}^{(t),k+1}_i+\tilde{v}^{(t),k}_j-D^{(t)}_{ij}}{\varepsilon}\right)\!\right]_j\right)}, ~~ t = 1, \cdots, N, \\
\Pi^{(t),k+1} &= {\textstyle\exp\!\left(\frac{\tilde{\bm{u}}^{(t),k+1}\bm{e}^{\top}_{m_t}
+\bm{e}_m(\tilde{\bm{v}}^{(t),k+1})^{\top}-D^{(t)}}{\varepsilon}\right)}, ~~ t = 1, \cdots, N,\\
\tilde{\bm{w}}^{k+1} &= \frac{\varepsilon}{N}{\textstyle\sum\limits^N_{t=1}\log\left(\!\left[\sum\limits^{m_t}_{j=1}
\exp\!\left(\frac{\tilde{u}^{(t),k+1}_i+\tilde{v}^{(t),k+1}_j-D^{(t)}_{ij}}{\varepsilon}\right)\!\right]_i\right)}, \\
\end{aligned}
\end{equation}
where $\tilde{\bm{w}}^0=\varepsilon\log(\frac{1}{m}\bm{e}_m)$ and $\tilde{\bm{u}}^{(t),0}=0$, $\tilde{\bm{v}}^{(t),0}=0$ for $t=1,\cdots,N$. After obtaining $\tilde{\bm{w}}^{k+1}$, one can recover $\bm{w}^{k+1}$ by setting $\bm{w}^{k+1}:=\exp\big{(}\tilde{\bm{w}}^{k+1}/\varepsilon\big{)}$. In contrast to \eqref{iterschemeIBP}, the log-domain iterations \eqref{iterschemeIBPlog} is more stable for a small $\varepsilon$. However, at each step, \eqref{iterschemeIBPlog} requires additional exponential operations that are typically time-consuming. It also loses some computational efficiency in replacing the matrix-vector multiplications (which can take advantage of the multiprocessing capability in {\sc Matlab}'s Intel Math Kernel Library) in \eqref{iterschemeIBP}
by the log-sum-exp operations. Hence, iterations \eqref{iterschemeIBPlog} can be much less efficient than iteration \eqref{iterschemeIBP} in computation. This issue has also been discussed in \citep[Remark 4.23]{pc2019computational}. Moreover, when $\varepsilon$ is small, the convergence of IBP can become quite slow. In our experiments, we use \eqref{iterschemeIBP} for $\varepsilon\in\{0.1,0.01\}$ and use \eqref{iterschemeIBPlog} for $\varepsilon=0.001$.

\section{A modified Bregman ADMM}\label{apdbadmm}

The Bregman ADMM (BADMM) was first proposed by \cite{wb2014bregman} and then was adapted to solve \eqref{subpro1} by \cite{ywwl2017fast}. For notational simplicity, let
\begin{equation*}
\begin{aligned}
\mathcal{C}_1 &:= \{(\Pi^{(1)}, \cdots, \Pi^{(N)})\,:\,(\Pi^{(t)})^{\top}\bm{e}_m = \bm{a}^{(t)}, ~\Pi^{(t)} \geq 0, ~t = 1, \cdots, N\},  \\
\mathcal{C}_2 &:= \{(\Gamma^{(1)}, \cdots, \Gamma^{(N)}, \bm{w})\,:\,\bm{w} \in \Delta_m, ~\Gamma^{(t)}\bm{e}_{m_t} = \bm{w}, ~\Gamma^{(t)} \geq 0, ~t=1,\cdots,N \}.
\end{aligned}
\end{equation*}
Then, problem \eqref{subpro1} can be equivalently rewritten as
\begin{equation}\label{subpro2}
\begin{aligned}
&\min\limits_{\{\Pi^{(t)}\},\,\{\Gamma^{(t)}\},\,\bm{w}}~{\textstyle \sum^N_{t=1}} \langle D^{(t)}, \,\Pi^{(t)} \rangle  \\
&\hspace{0.9cm} \mathrm{s.t.} \hspace{1cm} \Pi^{(t)} = \Gamma^{(t)}, \quad t = 1, \cdots, N, \\
&\hspace{2.4cm} (\Pi^{(1)},\cdots,\Pi^{(N)})\in\mathcal{C}_1, \quad (\Gamma^{(1)},\cdots,\Gamma^{(N)},\bm{w})\in\mathcal{C}_2.
\end{aligned}
\end{equation}
The iterative scheme of BADMM for solving \eqref{subpro2} is given by
\begin{equation*}
\left\{\begin{aligned}
&(\Pi^{(1),k+1},\cdots,\Pi^{(N),k+1}) = \mathop{\mathrm{argmin}}
\limits_{(\Pi^{(1)},\cdots,\Pi^{(N)})\in\mathcal{C}_1}\left\{\sum^N_{t=1}\left(\langle D^{(t)}, \,\Pi^{(t)} \rangle + \langle \Lambda^{(t),k}, \,\Pi^{(t)}\rangle + \rho\mathbf{KL}(\Pi^{(t)}, \,\Gamma^{(t),k})\right)\right\},  \\
&(\Gamma^{(1),k+1},\cdots,\Gamma^{(N),k+1},\bm{w}^{k+1}) = \mathop{\mathrm{argmin}}
\limits_{(\Gamma^{(1)},\cdots,\Gamma^{(N)},\bm{w})\in\mathcal{C}_2}\left\{\sum^N_{t=1}\left(-\langle\Lambda^{(t),k}, \,\Gamma^{(t)}\rangle + \rho\mathbf{KL}(\Gamma^{(t)}, \,\Pi^{(t),k+1})\right)\right\},\\
&\Lambda^{(t),k+1} = \Lambda^{(t),k} + \rho(\Pi^{(t),k+1} - \Gamma^{(t),k+1}), \quad t = 1,\cdots,N,
\end{aligned}\right.
\end{equation*}
where $\mathbf{KL}(\cdot, \cdot)$ denotes the KL divergence defined by $\mathbf{KL}(A, B)=\sum_{ij}a_{ij}\ln(\frac{a_{ij}}{b_{ij}})$ for any two matrices $A$, $B$ of the same size. The subproblems in above scheme have closed-form solutions; see \citep[Section III.B]{ywwl2017fast} for more details. Indeed, at the $k$-th iteration,
\begin{equation*}
\begin{aligned}
\bm{u}^{(t),k} &= \left(\frac{a^{(t)}_{j}}{(\Gamma^{(t),k}_{:j})^{\top}\exp(-{\textstyle\frac{1}{\rho}} D^{(t)}_{:j} - {\textstyle\frac{1}{\rho}}\Lambda^{(t),k}_{:j})}\right)_{j=1,\cdots,m_t}, \quad t = 1, \cdots, N, \\
\Pi^{(t),k+1} &= \left(\Gamma^{(t),k} \odot \exp(-{\textstyle\frac{1}{\rho}} D^{(t)} - {\textstyle\frac{1}{\rho}}\Lambda^{(t),k})\right)\mathrm{Diag}(\bm{u}^{(t),k}), \quad t = 1, \cdots, N, \\
\tilde{\bm{w}}^{(t),k+1} &= \left((\Pi^{(t),k+1}_{i:})^{\top}\exp({\textstyle\frac{1}{\rho}}\Lambda^{(t),k}_{i:})\right)_{i=1,\cdots,m}, \quad t = 1, \cdots, N, \\
\bm{w}^{k+1} &= \left({\textstyle\prod^N_{t=1}}\tilde{\bm{w}}^{(t),k+1}\right)^{\frac{1}{N}}\Big{/}
\left(\bm{e}_m^{\top}\left({\textstyle\prod^N_{t=1}}\tilde{\bm{w}}^{(t),k+1}\right)^{\frac{1}{N}}\right), \\
\bm{v}^{(t),k+1} &= \left(\frac{w^{k+1}_{i}}{(\Pi^{(t),k+1}_{i:})^{\top}\exp({\textstyle\frac{1}{\rho}}\Lambda^{(t),k}_{i:})}\right)_{i=1,\cdots,m}, \quad t = 1, \cdots, N, \\
\Gamma^{(t),k+1} &= \mathrm{Diag}(\bm{v}^{(t),k+1})\left(\Pi^{(t),k+1} \odot \exp({\textstyle\frac{1}{\rho}}\Lambda^{(t),k})\right), \quad t = 1, \cdots, N.
\end{aligned}
\end{equation*}
Moreover, in order to avoid computing the geometric mean $(\prod^N_{t=1}\tilde{\bm{w}}^{(t),k+1})^{\frac{1}{N}}$ for updating $\bm{w}^{k+1}$, \cite{ywwl2017fast} actually use one of the following heuristic rules to update $\bm{w}^{k+1}$:
\begin{equation*}
\begin{aligned}
(\mathrm{R1}) \qquad \bm{w}^{k+1} &= \left({\textstyle\sum^N_{t=1}}\tilde{\bm{w}}^{(t),k+1}\right)\Big{/}
\left(\bm{e}_m^{\top}\left({\textstyle\sum^N_{t=1}}\tilde{\bm{w}}^{(t),k+1}\right)\right), \\
(\mathrm{R2}) \qquad \bm{w}^{k+1} &= \left({\textstyle\sum^N_{t=1}}\sqrt{\tilde{\bm{w}}^{(t),k+1}}\right)^2\Big{/}
\left(\bm{e}_m^{\top}\left({\textstyle\sum^N_{t=1}}\sqrt{\tilde{\bm{w}}^{(t),k+1}}\right)^2\right).
\end{aligned}
\end{equation*}
In their {\sc Matlab} codes, (R2) is the default updating rule. The main computational complexity without considering the exponential operations in BADMM is $\mathcal{O}(m\sum^N_{t=1}m_t)$. For the exponential operations at each step,
the practical computational cost could be a few times more than the  previous cost of $\mathcal{O}(m\sum^N_{t=1}m_t)$.

\vskip 0.2in
\bibliography{Ref_WB}

\end{document}